\documentclass{amsart}

\usepackage[export]{ adjustbox }
\usepackage{ amsmath }
\usepackage{ amssymb }
\usepackage{ amsthm }
\usepackage{ booktabs }
\usepackage{ graphicx }
\usepackage{ latexsym }
\usepackage{ pxfonts }
\usepackage{ qtree }
\usepackage{ url }

\usepackage[ usenames ]{ color }
\usepackage[ flushleft, alwaysadjust ]{ paralist }
\usepackage[ all ]{ xy }

\theoremstyle{definition}
\newtheorem{definition}{Definition}[section]
\newtheorem{conjecture}[definition]{Conjecture}
\newtheorem{example}[definition]{Example}

\newtheorem{problem}[definition]{Problem}
\newtheorem{remark}[definition]{Remark}
\newtheorem{algorithm}[definition]{Algorithm}

\theoremstyle{plain}
\newtheorem{lemma}[definition]{Lemma}
\newtheorem{proposition}[definition]{Proposition}
\newtheorem{theorem}[definition]{Theorem}

\newcommand{\myfont}[1]{{\textnormal{\small\texttt{#1}}}}

\newcommand{\SO}{\Sigma O}

\newcommand{\BBB}{\myfont{BBB}}
\newcommand{\BW}{\myfont{BW}}
\newcommand{\BWS}{\myfont{BWS}}

\newcommand{\Assoc}{\myfont{Assoc}}
\newcommand{\SAssoc}{\Sigma{\myfont{Assoc}}}
\newcommand{\ComAss}{\myfont{ComAss}}
\newcommand{\DiAss}{\myfont{DiAss}}
\newcommand{\SDiAss}{\Sigma{\myfont{DiAss}}}

\newcommand{\TriAss}{\myfont{TriAss}}
\newcommand{\STriAss}{\Sigma{\myfont{TriAss}}}
\newcommand{\ComTriAss}{\myfont{ComTriAss}}

\newcommand{\Dend}{\myfont{Dend}}
\newcommand{\SDend}{\Sigma{\myfont{Dend}}}

\newcommand{\TriDend}{\myfont{TriDend}}
\newcommand{\STriDend}{\Sigma{\myfont{TriDend}}}
\newcommand{\ComTriDend}{\myfont{ComTriDend}}

\newcommand{\Lie}{\myfont{Lie}}

\newcommand{\Leib}{\myfont{Leib}}
\newcommand{\Zinb}{\myfont{Zinb}}
\newcommand{\TriLie}{\myfont{TriLie}}
\newcommand{\Perm}{\myfont{Perm}}
\newcommand{\PreLie}{\myfont{PreLie}}
\newcommand{\PostLie}{\myfont{PostLie}}

\newcommand{\Jor}{\myfont{Jor}}
\newcommand{\DiJor}{\myfont{DiJor}}
\newcommand{\TriJor}{\myfont{TriJor}}
\newcommand{\PreJor}{\myfont{PreJor}}
\newcommand{\PostJor}{\myfont{PostJor}}

\newcommand{\VectF}{\mathbf{Vect}_\mathbb{F}}

\makeatletter
\def\l@section{\@tocline{1}{0pt}{1pc}{}{}}
\def\l@subsection{\@tocline{2}{0pt}{1pc}{4.6em}{}}
\def\l@subsubsection{\@tocline{3}{0pt}{1pc}{7.6em}{}}
\renewcommand{\tocsection}[3]{%
  \indentlabel{\@ifnotempty{#2}{\makebox[1.25em][l]{\ignorespaces#1#2.}}}#3}
\renewcommand{\tocsubsection}[3]{%
  \indentlabel{\@ifnotempty{#2}{\hspace*{1.25em}\makebox[2.00em][l]{\ignorespaces#1#2.}}}#3}
\renewcommand{\tocsubsubsection}[3]{%
  \indentlabel{\@ifnotempty{#2}{\hspace*{3.25em}\makebox[2.75em][l]{\ignorespaces#1#2.}}}#3}
\makeatother

\allowdisplaybreaks

\begin{document}

\title{Jordan Trialgebras and Post-Jordan Algebras}

\author{Fatemeh Bagherzadeh}

\address{Department of Mathematics and Statistics, University of Saskatchewan, Canada}

\email{bagherzadeh@math.usask.ca}

\author{Murray Bremner}

\address{Department of Mathematics and Statistics, University of Saskatchewan, Canada}

\email{bremner@math.usask.ca}

\author{Sara Madariaga}

\address{Department of Mathematics and Computer Science, University of La Rioja, Spain}

\email{sara.madariaga@unirioja.es}



\thanks{The authors thank Vladimir Dotsenko for numerous helpful discussions.
The authors were supported by a Discovery Grant from NSERC,
the Natural Sciences and Engineering Research Council of Canada}

\begin{abstract}
We compute minimal sets of generators for the $S_n$-modules ($n \le 4$) of multilinear polynomial identities
of arity $n$ satisfied by the Jordan product and the Jordan diproduct (resp.~pre-Jordan product) in every
triassociative (resp.~tridendriform) algebra.
These identities define Jordan trialgebras and post-Jordan algebras: Jordan analogues of the Lie trialgebras
and post-Lie algebras introduced by Dotsenko et al., Pei et al., Vallette \& Loday.
We include an extensive review of analogous structures existing in the literature, and their interrelations,
in order to identify the gaps filled by our two new varieties of algebras.
We use computer algebra (linear algebra over finite fields, representation theory of symmetric groups), to verify
in both cases that every polynomial identity of arity $\le 6$ is a consequence of those of arity $\le 4$.
We conjecture that in both cases the next independent identities have arity 8, imitating the Glennie identities
for Jordan algebras.
We formulate our results as a commutative square of operad morphisms, which leads to the conjecture that
the squares in a much more general class are also commutative.
\end{abstract}

\maketitle


\vspace{-6mm}

{
\small
\tableofcontents
}


\section{Introduction}


\subsection{Operadic generalization of Lie and Jordan algebras}

Lie and Jordan algebras are defined by the polynomial identities of arity $n \le 3, 4$ satisfied
by the (anti)commutator in every associative algebra.
Lie dialgebras (Leibniz algebras) were introduced in the early 1990s \cite{Loday1993} together with
diassociative algebras and the (anti)dicommutator; Jordan dialgebras (quasi-Jordan algebras)
appeared 10 years later \cite{VF}.
Dendriform algebras, governed by the Koszul dual of the diassociative operad, were introduced in the late 1990s
\cite{LR}; in this case, the (anti)dicommutator produces pre-Lie
\cite{Gerstenhaber,Vinberg} and pre-Jordan \cite{HNB} algebras.
These constructions stimulated the theory of algebraic operads and were reformulated in terms
of Manin white and black products.
It was then a short step to triassociative and tridendriform algebras, and their Lie analogues,
Lie trialgebras \cite{DG} and post-Lie algebras \cite{Vallette1}.

In this paper, we investigate the Jordan analogues, and define Jordan trialgebras and post-Jordan algebras%
\footnote{One choice of defining relations for post-Jordan algebras appears in Appendix A.5 of the \url{arXiv}
version of Bai et al.~\cite{BBGN}, where these relations are presented as defining the trisuccessor of the
Jordan operad.  That Appendix does not appear in the published version of the paper, and in any case, our
methods to obtain them are entirely different.}.
For $X \in \{ \mathrm{Lie}, \mathrm{Jordan} \}$, $X$-trialgebras combine $X$-algebras and $X$-dialgebras in
one structure; post-$X$ algebras combine $X$-algebras and pre-$X$ algebras in one structure.
Our methods are primarily computational; we determine the multilinear identities of arity $\le 6$ satisfied by
the anticommutator and antidicommutator in every triassociative and tridendriform algebra.
We use combinatorics of trees, linear algebra over finite fields, and representation theory of
symmetric groups.
The identities form the nullspace of what we call expansion matrix, which represents (with respect to
ordered monomial bases) a morphism from an operad of Jordan type to one of associative type.
The defining identities have arity $\le 4$ in both cases, and we verify that there
are no new identities of arity 5 or 6.


\begin{table}[ht]
\small
\boxed{
\begin{tabular}{l|l}
\textbf{ASSOCIATIVE ALGEBRAS}
&
Self-dual
\\[-1pt]
Operation $ab$ Relation $(ab)c \equiv a(bc)$
\\ \midrule
Lie bracket $[a,b] = ab - ba$
&
symmetry $ab \equiv ba$, relation $(ab)c \equiv a(bc)$
\\[-1pt]
\textbf{Lie algebras}, equation \eqref{lieidentities}
&
\textbf{Commutative associative algebras}
\\ \midrule
Jordan product $a \,\circ\, b = ab + ba$
\\[-1pt]
\textbf{Jordan algebras}, equation \eqref{jordanidentities}
&
No dual, operad is cubic not quadratic
\\ \midrule
\textbf{DIASSOCIATIVE ALGEBRAS}
&
\textbf{DENDRIFORM ALGEBRAS}
\\[-1pt]
Operations $a \dashv b$, $a \vdash b$, Definition \ref{diassdendrdef}
&
Operations $a \prec b$, $a \succ b$, Definition \ref{diassdendrdef}
\\ \midrule
Leibniz bracket $\{a,b\} = a \dashv b - b \vdash a$
\\[-1pt]
\textbf{Leibniz algebras}, Definition \ref{leibnizdef}
&
\textbf{Zinbiel algebras}, Definition \ref{leibnizdef}
\\ \midrule
&
Pre-Lie product $\{ a, b \} = a \prec b - b \succ a$
\\[-1pt]
\textbf{Perm algebras}, Definition \ref{preliedef}
&
\textbf{Pre-Lie algebras}, Definition \ref{preliedef}
\\ \midrule
Jordan diproduct $a \,\circ\, b = a \dashv b + b \vdash a$
\\[-1pt]
\textbf{Jordan dialgebras}, Definition \ref{jordialgdef}
&
No dual, operad is cubic not quadratic
\\ \midrule
&
Pre-Jordan product $a \bullet b = a \prec b + b \succ a$
\\[-1pt]
No dual, operad is cubic not quadratic
&
\textbf{Pre-Jordan algebras},
Definition \ref{jordialgdef}
\\
\midrule
\textbf{TRIASSOCIATIVE ALGEBRAS}
&
\textbf{TRIDENDRIFORM ALGEBRAS}
\\[-1pt]
Operations $a \dashv b$, $a \perp b$, $a \vdash b$
&
Operations $a \prec b$, $a \curlywedge b$, $a \succ b$
\\[-1pt]
Definition \ref{triassdef}
&
Definition \ref{triassdef}
\\ \midrule
Lie bracket $[a,b] = a \perp b - b \perp a$
&
\\[-1pt]
Leibniz bracket $\{ a, b \} = a \dashv b - b \vdash a$
&
\textbf{Commutative tridendriform algebras},
\\[-1pt]
\textbf{Lie trialgebras}, Definition \ref{comtridendef}
&
Definition \ref{comtridendef}
\\ \midrule
&
Lie bracket $[a,b] = a \curlywedge b - b \curlywedge a$
\\[-1pt]
\textbf{Commutative triassociative algebras},
&
Pre-Lie product $\{ a, b \} = a \prec b - b \succ a$
\\[-1pt]
Definition \ref{comtriassdef}
&
\textbf{Post-Lie algebras}, Definition \ref{comtriassdef}
\\ \midrule
Jordan product $a \,\circ\, b = a \perp b + b \perp a$
\\[-1pt]
Jordan diproduct $a \bullet b = a \dashv b + b \vdash a$
\\[-1pt]
\textbf{Jordan trialgebras}, Section \ref{jordantrialgebrasection}
&
No dual, operad is cubic not quadratic
\\ \midrule
&
Jordan product $a \,\circ\, b = a \curlywedge b + b \curlywedge a$
\\[-1pt]
&
Pre-Jordan product $a \bullet b = a \prec b + b \succ a$
\\[-1pt]
No dual, operad is cubic not quadratic
&
\textbf{Post-Jordan algebras}, Section \ref{postjordanalgebrasection}
\end{tabular}
}
\medskip
\caption{Operadic generalizations of associative, Lie, and Jordan algebras}
\vspace{-6mm}
\label{varietiesofnonassociativealgebras}
\end{table}


\subsection{Overview of problems and methods}

Table \ref{varietiesofnonassociativealgebras} displays the generalizations of associativity underlying our
results; the number of operations increases from top to bottom.
The left column contains operads obtained from associative operations (di- and triassociative); the right column
contains their Koszul duals (dendriform and tridendriform).
These operads are nonsymmetric (but will be symmetrized); each of them leads to generalizations of Lie and
Jordan algebras, defined by the polynomial identities of arity $\le 3, 4$ satisfied by the analogues of
the (anti)commutator and (anti)dicommutator.
All the operads we consider are generated by binary operations, with or without symmetry.

\subsubsection{The symmetric operads $\STriAss$, $\STriDend$ and $\BW$}

We are primarily concerned with three operads; the first two form a Koszul dual pair:
\begin{itemize}
\item
The symmetric triassociative operad $\STriAss$ (the symmetrization of $\TriAss$)
which is generated by three binary operations $\dashv, \perp, \vdash$ with no symmetry, satisfying seven quadratic relations; see Definition \ref{triassdef}.
\item
The symmetric tridendriform operad $\STriDend$ (the symmetrization of $\TriDend$) which
is generated by three binary operations  $\prec, \curlywedge, \succ$ with no symmetry, satisfying 11 quadratic relations; see Definition \ref{triassdef}.
\item
The free symmetric operad $\BW$ generated by two binary operations, one commutative $\circ$ and one
noncommutative (with no symmetry) $\bullet$.
The operads governing Jordan trialgebras and post-Jordan algebras are quotients of $\BW$.
\item
The expansion map $E(n)\colon \BW(n) \to\STriAss(n)$ is defined by
$a \,\circ\, b \mapsto a \perp b + b \perp a$
and
$a \bullet b \mapsto a \dashv b + b \vdash a$;
its kernel contains the defining identities for Jordan trialgebras.
\item
The expansion map $E(n)\colon \BW(n) \to\STriDend(n)$ is defined by
$a \,\circ\, b \mapsto a \curlywedge b + b \curlywedge a$
and
$a \bullet b \mapsto a \prec b + b \succ a$;
its kernel contains the defining identities for post-Jordan algebras.
\end{itemize}
The standard monograph on the theory of algebraic operads is Loday \& Vallette \cite{LV}.


\subsubsection{Computational methods}

All computations are performed with Maple using arithmetic over $\mathbb{Z}$ or $\mathbb{Q}$ or a finite field.
To save time and memory we usually work over a finite prime field $\mathbb{F}_p$,
where $p$ is greater than the arity $n$ of the identities being considered;
this guarantees that the group algebra $\mathbb{F}_p S_n$ is semisimple.
In arity $n$, both the domain and the codomain of the expansion map are $S_n$-modules,
and so if $p > n$ then semisimplicity of $\mathbb{F}_p S_n$
guarantees that it is isomorphic to the direct sum of simple two-sided ideals each of them isomorphic to a full matrix algebra.
The coefficients in the formulas for the matrix units in $\mathbb{F}_p S_n$ as
linear combinations of permutations have denominators which are divisors of $n!$ and hence are well-defined modulo $p$.
We are left with the problem of reconstructing correct rational results from modular
calculations, but in this respect we had very good luck: all matrix entries are in $\{ 0, \pm 1 \}$.
Using the representation theory of the symmetric group allows us to decompose a large matrix into much smaller pieces.


\subsubsection{Terminology and notation}

The number of arguments in a monomial is often called its \emph{degree}; here we use \emph{arity}:
for algebraic operads, \emph{degree} refers to (homological) degree in a differential graded vector space.
From the homological point of view, all our vector spaces are graded but concentrated in degree 0
over a base field $\mathbb{F}$ of characteristic 0, unless otherwise stated.
We write $\VectF$ for the category of vector spaces over $\mathbb{F}$.
If $X$ is a set then $\mathbb{F}X$ is the vector space%
\footnote{From a categorical perspective, the functor $X \mapsto \mathbb{F}X$ is left adjoint to
the forgetful functor $\VectF \to \mathbf{Sets}$; that is, $\mathbb{F}X$ is the free vector space
over $\mathbb{F}$ generated by $X$.}
with basis $X$.
If $\mathcal{O}$ is a nonsymmetric operad then $\Sigma\mathcal{O}$ denotes its symmetrization.


\subsection{Associative, Lie, and Jordan algebras}

An associative algebra is a vector space $A$ with a bilinear product $m\colon A \times A \to A$,
$(a,b) \mapsto ab$, satisfying the relation $(ab)c - a(bc) \equiv 0$ which is nonsymmetric
(hence multilinear): every term has the identity permutation of the arguments.
The symbol $\equiv$ indicates that equality holds for all values of the arguments.

In every associative algebra, the commutator (Lie bracket) $[a,b] = ab - ba$ is antisymmetric and
satisfies the Jacobi identity; these properties define Lie algebras:
  \begin{equation}
  \label{lieidentities}
  [a,a] \equiv 0, \qquad [[a,b],c] + [[b,c],a] + [[c,a],b] \equiv 0.
  \end{equation}
In every associative algebra, the anticommutator (Jordan product) $a \,\circ\, b = ab + ba$ is
symmetric and satisfies the Jordan identity; these properties define Jordan algebras:
  \begin{equation}
  \label{jordanidentities}
  a \,\circ\, b \equiv b \,\circ\, a,
  \qquad
  ( ( a \,\circ\, a ) \,\circ\, b ) \,\circ\, a \equiv ( a \,\circ\, a ) \,\circ\, ( b \,\circ\, a ).
  \end{equation}
Let $L$ be a Lie algebra with bracket $[-,-]$, let $J$ be a Jordan algebra with product $a \circ b$,
and let $X$ be either $L$ or $J$.
Then $X$ has a universal associative enveloping algebra $U(X)$ in the following sense:
if $f\colon X \to A$ is a linear map to an associative algebra $A$ then there is a unique algebra morphism
$g\colon U(X) \to A$ such that $g \,\circ\, h = f$ where $h\colon X \to U(X)$ is the natural map arising
as follows.
If $T(V)$ is the tensor algebra with product $a \cdot b$ of the vector space $V$,
then $U(X) \cong T(X) / I(X)$ for the following (two-sided) ideals:

\vspace{1mm}

\centerline{
$I(L) = \langle \, a \cdot b - b \cdot a - [a,b] \mid a, b \in L \, \rangle$,
\quad\quad
$I(J) = \langle \, a \cdot b + b \cdot a - a \,\circ\, b \mid a, b \in J \, \rangle$.
}

\vspace{1mm}

\noindent
For every Lie algebra $L$, $h$ is injective; this follows from the Poincar\'e-Birkhoff-Witt theorem
and implies that every polynomial identity satisfied by the commutator in every associative algebra is
a consequence of skewsymmetry and the Jacobi identity.
The same does not hold for Jordan algebras.
A Jordan algebra $J$ is \emph{special} if it is isomorphic to a subalgebra of $A^+$ for some associative
algebra $A$, where $A^+$ is the same vector space with the nonassociative product $a \,\circ\, b = ab + ba$.
For a Jordan algebra $J$, the natural map $h\colon J \to U(J)$ is injective if and only if $J$ is special.
There are polynomial identities satisfied by every special Jordan algebra which are not consequences
of symmetry and the Jordan identity \cite{Glennie}; the smallest examples are the so-called Glennie identities and occur in arity 8.
Quotients of special Jordan algebras are not necessarily special \cite{Cohn}, so
special Jordan algebras do not form a variety defined by polynomial identities.
\begin{figure}[h]
\centering
\includegraphics[scale=1]{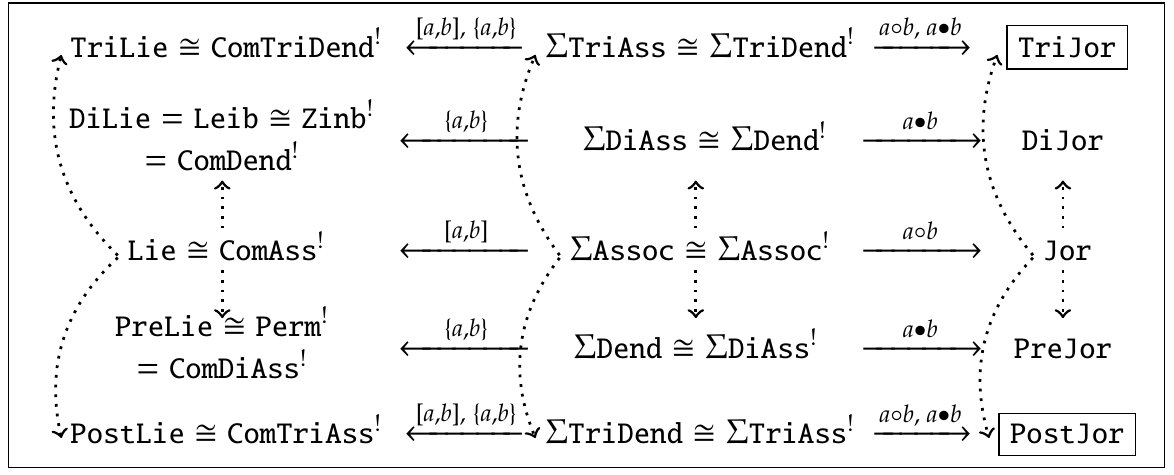}
\vspace{-4mm}
\caption{Generalizations of the Lie, associative, and Jordan operads}
\vspace{-2mm}
\label{generalizingoperads}
\end{figure}

\subsection{Results of this paper in context}

In later sections, we include concise reviews of analogous structures existing in the literature, and their interrelations,
in order to identify the gaps which motivated the writing of this paper.
Table \ref{generalizingoperads} shows the algebraic operads with which we are concerned.
Our results provide definitions of Jordan trialgebras and post-Jordan algebras which are indicated by boxes.
(A somewhat similar diagram appears in the lecture notes \cite{Kolesnikov2}.)
The dotted straight (resp.~curved) up-down arrows indicate the white (resp-~black) Manin products with the operads
$\Perm$ and $\PreLie$ (resp.~$\ComTriAss$ and $\PostLie$); see \S\ref{sectionManin}.
These procedures may also be described as duplicators and disuccessors (resp. triplicators and trisuccessors);
see \S\ref{triplicatorsection}, \S\ref{trisuccessorsection}.
Reversing an arrow corresponds to realizing a simpler structure as a subalgebra or quotient of a more
complex structure.
The left-right arrows represent morphisms between operads which replace the (generalized) associative product
by one or two (generalized) Lie brackets or Jordan products.
We use the symbols $[a,b]$ and/or $\{a,b\}$ for the Lie case,
$a \circ b$ and/or $a \bullet b$ for Jordan.
Reversing a left-right arrow represents constructing the universal enveloping algebra of the
corresponding Lie or Jordan structure.
All these operads are binary: the generating operations are bilinear.
In columns 1, 2 the operads are quadratic: the terms of the relations contain two operations.
In column 3, the operads are cubic: the terms contain three operations.
Koszul duality for quadratic operads can be applied to columns 1, 2 but not 3.


\section{Preliminaries on algebraic operads}


\subsection{Free nonsymmetric binary operads}

Let $\Omega = \{ \circ_1, \dots, \circ_m \}$ be a finite ordered set of binary operations;
let $O_m$ be the free nonsymmetric operad in $\VectF$ generated by $\mathbb{F}\Omega$.
The underlying vector space of $O_m$ is the direct sum over $n \ge 1$ of homogeneous components $O_m(n)$
of arity $n$.
We never add elements of different arities, so we could replace the direct sum by a disjoint union
(but then we would leave $\VectF$).
The following construction of $O_m$ also makes sense in the category \textbf{Set} of sets.
We define sets $X$ of operations of arity $n$ and then consider the vector spaces $\mathbb{F}X$.
Disjoint union (resp.~direct product) of sets corresponds to direct sum (resp.~tensor product) of vector spaces.
The operad $O_m$ in \textbf{Set} may be identified with the free magma on one generator with $m$ binary operations.

\begin{definition}
For $n \ge 1$, let $P_n = \{ p_1, \dots, p_{c(n)} \}$ be the \textbf{association types} of arity $n$:
the distinct placements of balanced pairs of parentheses for a single binary operation in a sequence of
$n$ arguments.
If we omit the outermost pair then for $n \ge 2$ there are $n{-}2$ pairs.
The size of $P_n$ is the Catalan number $c(n) = (2n{-}2)!/(n!(n{-}1)!)$.
We write $P = \bigcup_{n \ge 1} P_n$; the set $P$ is a basis for $\mathcal{O}_1$, the free nonsymmetric operad
on one binary operation.
\end{definition}

\begin{example} \label{exampleplacement}
We write $x$ for the arguments and juxtaposition for the operation.
For $n = 1$ we have only $x$, for $n = 2$ only $xx$, and for $n = 3$ only $(xx)x$, $x(xx)$.
For $n = 4, 5$ we have:
\[
\begin{array}{c@{\qquad}l}
n = 4 &
((xx)x)x, \;\; (x(xx))x, \;\; (xx)(xx), \;\; x((xx)x), \;\; x(x(xx)).
\\
n = 5 &
(((xx)x)x)x, \;\; ((x(xx))x)x, \;\; ((xx)(xx))x, \;\; (x((xx)x))x, \;\; (x(x(xx)))x,
\\
&
((xx)x)(xx), \;\; (x(xx))(xx), \;\; (xx)((xx)x), \;\; (xx)(x(xx)),
\\
&
x(((xx)x)x), \;\; x((x(xx))x), \;\; x((xx)(xx)), \;\; x(x((xx)x)), \;\; x(x(x(xx))).
\end{array}
\]
\end{example}

\begin{definition}
The \textbf{revdeglex} order $\prec$ on $P$ is defined as follows.
If $f \in P_n$, $f' \in P_{n'}$, $n < n'$ then $f \prec f'$.
If $f, f' \in P_n$ then we proceed by induction.
If $n \le 2$ then $| P_n | = 1$ and $f = f'$.
If $n \ge 3$ then there are unique factorizations $f = gh$, $f' = g'h'$;
we set $f \prec f'$ if and only if either $h \prec h'$, or $h = h'$, $g \prec g'$.
(This order was used in Example \ref{exampleplacement}.)
\end{definition}

If $m \ge 2$ then we must distinguish the $m$ operation symbols in $\Omega$.
An element of $P_n$ contains $n{-}1$ multiplications ordered from left to right,
and each may be replaced by any element of $\Omega$.
We identify the $m^{n-1}$ possibilities with $\Omega^{n-1}$ whose lex order $\prec_{\Omega^{n-1}}$ is
induced by $\Omega$: if $\alpha, \beta \in \Omega^{n-1}$ then $\alpha \prec_{\Omega^{n-1}} \beta$ if and only if
$\alpha_i \prec_\Omega \beta_i$ where $i$ is minimal for $\alpha_i \ne \beta_i$.

\begin{definition}
Operations are independent of association types, and so we identify a basis of $O_m(n)$ with
$P_n^{(m)} = P_n \,\times\, \Omega^{n-1}$.
The elements of $P_n^{(m)}$ are the \textbf{skeletons} of arity $n$; each skeleton is an $n$-tuple
$( \, p_i; \, {\circ}_{j_1}, \, \dots, \, {\circ}_{j_{n-1}} \, )$
which represents the basis element of $O_m(n)$ with association type $p_i$
and operations ${\circ}_{j_1}, \, \dots, \, {\circ}_{j_{n-1}}$ from left to right.
The total orders on $P_n$ and $\Omega^{n-1}$ extend to $P_n^{(m)}$:
for $p, p' \in P_n$ and $\alpha, \alpha' \in \Omega^{n-1}$ we have $( p, \alpha ) \prec ( p', \alpha' )$
if and only if either $p \prec_P p'$, or $p = p'$, $\alpha \prec_{\Omega^{n-1}} \alpha'$.
\end{definition}

\begin{definition}
\label{composeskeletons}
To avoid confusion with the operations in $\Omega$, we denote the \textbf{substitution maps} in $O_m(n)$ by an asterisk
instead of a circle,
$\ast_i\colon O_m(n) \otimes O_m(n') \rightarrow O_m(n{+}n'{-}1)$ for $1 \le i \le n$.
For skeletons $f \in P_n^{(m)} \subseteq O_m(n)$ and $f' \in P_{n'}^{(m)} \subseteq O_m(n')$,
the skeleton
\[
f \ast_i f' \in P_{n+n'-1}^{(m)} \subseteq O_m(n{+}n'{-}1),
\]
is obtained by substituting $f'$ for the $i$-th argument of $f$.
These definitions extend bilinearly to $O_m$, and the resulting substitutions generate all compositions in $O_m$.
\end{definition}

\begin{definition}
\label{nonsymdefinition}
We may now state precisely that the \textbf{free nonsymmetric operad} with $m$ binary operations $\Omega$
with no symmetry is the direct sum $O_m$ of the finite dimensional vector spaces $O_m(n)$ with
the compositions generated by the substitutions $\ast_i$ of Definition \ref{composeskeletons}.
\end{definition}


\subsection{Free symmetric binary operads}

The free nonsymmetric operad $O_m$ has a basis of skeletons containing parentheses and operation symbols;
arguments are denoted by $x$.

\begin{definition}
A \textbf{multilinear monomial} of arity $n$ is a skeleton $s \in P_n^{(m)}$ in which the argument symbols
have been replaced by a permutation of $x_1, \dots, x_n$.
We identify the set of such monomials with $W_n^{(m)} = P_n^{(m)} \times S_n$ where $S_n$ is the symmetric group,
and define $\SO_m(n)$ to be the vector space with basis $W_n^{(m)}$ on which there is a right $S_n$-action.
The direct sum $\SO_m$ of the homogeneous components $\SO_m(n)$ is the \textbf{symmetrization} of $O_m$.
In fact, $\SO_m$ is the (underlying vector space of the) \textbf{free symmetric operad} on $m$ binary operations
with no symmetry.
Definition \ref{composeskeletons} shows how to compose skeletons, but in $\SO_m$ we also need to compose monomials:
if $\alpha, \alpha'$ are monomials of arities $n, n'$ with arguments $x_1, \dots, x_n$ and $x_1, \dots, x_{n'}$,
then $\alpha \ast_i \alpha'$ ($1 \le i \le n$) is a monomial of arity $n{+}n'{-}1$ with arguments
$x_1, \dots, x_{n+n'-1}$.
For further information, see \S5.3.2 of \cite{LV}.
A \textbf{regular} operad is a symmetric operad (free or not) which is the symmetrization of a nonsymmetric operad.
\end{definition}


\subsection{Identical relations and operad ideals}

Let $P$ be a free symmetric operad generated by $P(2)$; that is, binary operations with or without symmetry.
The symmetries of the operations are determined by the $S_2$-module structure of $P(2)$, which is the direct
sum of $S_2$-modules isomorphic to either the unit module $[+]$ (representing commutativity),
the sign module $[-]$ (anticommutativity), or the group algebra $\mathbb{F} S_n$ (no symmetry).
(Since $\mathbb{F} S_n \cong [+] \oplus [-]$, an operation with no symmetry can be polarized into two operations,
one commutative and one anticommutative, but this will not concern us.)
We impose an $S_3$-module of quadratic relations $R \subseteq P(3)$ on the operations generating $P$.

\begin{definition}
The \textbf{operad ideal} $I = \langle R \rangle$ generated by the relations $R$ is the smallest sum
$\bigoplus_{n \ge 1} I(n)$ of $S_n$-submodules $I(n) \subseteq P(n)$ which is closed under arbitrary
compositions with elements of $P$.
That is, if $f \in I(r)$ and $g \in P(s)$ then $f \ast_i g \in I(r{+}s{-}1)$ for $1 \le i \le r$, and
$g \ast_j f \in I(r{+}s{-}1)$ for $1 \le j \le s$.
The \textbf{quotient operad} $P/I$ has homogeneous components $P(n)/I(n)$ with compositions defined
in the natural way.
\end{definition}

For given relations $R$, we construct a set of $S_n$-module generators of $I(n)$ for all $n$.
The $S_3$-module $R = I(3)$ is generated by a finite set of relations.
Assume that we have already constructed a finite set of $S_n$-module generators of $I(n)$.
If $f$ is such a generator then we increment the arity by composing $f( x_1, \dots, x_n )$
with some  $g( x_1, x_2 ) \in P(2)$.
By multilinearity, we may assume that $g( x_1, x_2 ) = x_1 \circ x_2$ where $\circ$ is one of the
generating operations.
There are $n$ possibilities for $f \ast_i g$ and two for $g \ast_j f$;
since every relation must be multilinear, we change subscripts of some arguments:
\begin{equation}
\label{operadcompositions}
\left\{ \;
\begin{array}{l}
f \ast_i g =
f( x_1, \dots, x_n ) \ast_i g( x_1, x_2 ) =
f( x_1, \dots, x_{i-1}, x_i \,\circ\, x_{n+1}, x_{i+1}, \dots, x_n )
\\[1mm]
g \ast_j f =
g( x_1, x_2 ) \ast_j f( x_1, \dots, x_n ) =
\left\{
\begin{array}{ll}
\!\!
f( x_1, \dots, x_n ) \,\circ\, x_{n+1} &\;\; (j = 1)
\\
\!\!
x_{n+1} \,\circ\, f( x_1, \dots, x_n ) &\;\; (j = 2)
\end{array}
\right.
\end{array}
\right.
\end{equation}
As $f$ runs over the set of $S_n$-module generators of $I(n)$, the elements \eqref{operadcompositions}
form a set of $S_{n+1}$-module generators for $I(n{+}1)$.


\subsection{Koszul duality for quadratic operads}

For any category of algebras governed by a quadratic operad, one can construct its Koszul dual operad, which
governs another category of algebras.
For associative algebras, the nonsymmetric operad is self-dual: $\Assoc^! \cong \Assoc$.
The symmetric operads $\ComAss$ for commutative associative algebras and $\Lie$ for Lie algebras
form a Koszul dual pair: $\ComAss^! = \Lie$, $\Lie^! = \ComAss$.
(If $P$ is a quadratic operad then $(P^!)^! \cong P$.)
If we generalize associativity to more than one binary operation (diassociative, triassociative, \dots),
the quadratic operads are not self-dual, and define further generalizations of associativity (dendriform,
tridendriform, \dots).

Loday showed that for an operad (symmetric or nonsymmetric) generated by binary operations
with no symmetry, the relations for the Koszul dual can be obtained from the original relations
by means of elementary linear algebra; see Theorem 8.5 and Proposition B.3 of \cite{Loday2001} and
Table 2 of \cite{BSO}.
In this paper we give two examples, one of which extends the algorithm to operations with symmetry;
see Lemmas \ref{TATDduality} and \ref{lietriops}.


\subsection{Diassociative and dendriform algebras}

\begin{definition}
\label{diassdendrdef}
\cite{Loday1995a}, \cite{LR}, \cite{Loday2001}.
A vector space with bilinear operations $\dashv$, $\vdash$ is a \textbf{diassociative algebra} if it satisfies these relations:
\begin{itemize}
\item
left and right associativity: \;
$( a \dashv b ) \dashv c \equiv a \dashv ( b \dashv c )$, \;
$( a \vdash b ) \vdash c \equiv a \vdash ( b \vdash c )$,
\item
inner associativity: \;
$( a \vdash b ) \dashv c \equiv a \vdash ( b \dashv c )$,
\item
left and right bar identities: \;
$( a \dashv b ) \vdash c \equiv ( a \vdash b ) \vdash c$, \;
$a \dashv ( b \dashv c ) \equiv a \dashv ( b \vdash c )$.
\end{itemize}
A vector space with bilinear operations $\prec$, $\succ$ is a \textbf{dendriform algebra}
if it satisfies:
\begin{itemize}
\item
inner associativity: \;
$( a \succ b ) \prec c \equiv a \succ ( b \prec c )$,
\item
left-right symmetrization: \;
$( a \prec b ) \prec c \equiv a \prec ( b \prec c ) + a \prec ( b \succ c )$,
\item
right-left symmetrization: \;
$a \succ ( b \succ c ) \equiv ( a \succ b ) \succ c + ( a \prec b ) \succ c$.
\end{itemize}
\end{definition}

\begin{remark}
The diassociative (but not dendriform) relations have the form $m_1 \equiv m_2$ for monomials $m_1, m_2$.
Dendriform algebras are related to Rota-Baxter operators \cite{Aguiar}, \cite{Ebrahimi-Fard}.
\end{remark}

\begin{lemma}
\cite{LR}, \cite{Loday2001}.
The operation $a \cdot b = a \prec b + a \succ b$ is associative in every dendriform algebra.
The operads $\DiAss$ $(\SDiAss)$ and $\Dend$ $(\SDend)$ are Koszul duals.
\end{lemma}

\begin{definition}
\label{leibnizdef}
\cite{Bloh}, \cite{Loday1993}, \cite{Loday1995b}.
A vector space with a bilinear operation $\{a,b\}$ (resp.~$a \prec b$) is a (left) \textbf{Leibniz}
(resp.~\textbf{Zinbiel}) algebra if it satisfies the relation
  \[
  \{ a, \{ b, c \} \} \equiv \{ \{ a, b \}, c \} + \{ b, \{ a, c \} \}
  \;\; \big( \text{resp.} \;\;
  ( a \prec b ) \prec c \equiv a \prec ( b \prec c ) + a \prec ( c \prec b ) \big).
  \]
\end{definition}

\begin{lemma}
\cite{Loday1995a,Loday1995b,Loday2001}.
The operation $a \dashv b - b \vdash a$ satisfies the left Leibniz identity in every diassociative algebra.
In every Zinbiel algebra the antidicommutator $a \prec b + b \prec a$ is commutative and associative.
Every Zinbiel algebra becomes a dendriform algebra if we define $a \succ b = b \prec a$; conversely, every
dendriform algebra which is commutative ($a \prec b \equiv b \succ a$) is a Zinbiel algebra.
The operads $\Leib$ and $\Zinb$ are Koszul duals.
\end{lemma}

\begin{remark}
The operation $\{a,b\} = a \dashv b - b \vdash a$ defines a functor between the categories of diassociative and
Leibniz algebras; the left adjoint sends a Leibniz algebra to its universal enveloping diassociative algebra;
see \cite{AG}, \cite{BCL}, \cite{Goichot}, \cite{Loday2001}.
\end{remark}

\begin{problem}
Determine the polynomial identities satisfied by
(1)
the dicommutator and antidicommutator in Leibniz algebras;
(2)
the dicommutator in Zinbiel algebras.
\end{problem}

\begin{definition}
\label{preliedef}
\cite{Chapoton1}, \cite{CL}, \cite{Gerstenhaber}, \cite{Vinberg}.
A vector space with a bilinear operation $\{a,b\}$ is a (left) \textbf{pre-Lie} (or left-symmetric) algebra
if it satisfies the relation $( a, b, c ) \equiv ( b, a, c )$ where the associator is
$( a, b, c ) = \{ \{ a, b \}, c \} - \{ a, \{ b, c \} \}$.
A (right) \textbf{perm algebra} is associative and \textbf{right-commutative}: $a b c \equiv a c b$.
\end{definition}

\begin{lemma}
\cite{BL}, \cite{Chapoton2}, \cite{Loday1995a}, \cite{Loday2001}.
In every dendriform algebra the operation $a \prec b - b \succ a$ satisfies the pre-Lie identity.
The commutator in a pre-Lie algebra satisfies (anticommutativity and) the Jacobi identity.
Every identity for the anticommutator in a pre-Lie algebra is a consequence of commutativity.
The operads $\PreLie$ and $\Perm$ are Koszul duals.
\end{lemma}

\begin{definition}
\label{jordialgdef}
\cite{Bremner}, \cite{BM}, \cite{HNB}, \cite{Kolesnikov1}, \cite{VF}.
A (right) \textbf{Jordan dialgebra} is a vector space with a bilinear product $a \bullet b$
satisfying the following three relations (one of arity 3 and two of arity 4)
where $(a,b,c) = ( a \bullet b ) \bullet c - a \bullet ( b \bullet c )$ is the associator:
  \begin{center}
  \begin{tabular}{l@{\qquad}l}
  right commutativity & $a \bullet ( b \bullet c ) \equiv a \bullet ( c \bullet b )$
  \\
  right quasi-Jordan identity & $( b \bullet a ) \bullet ( a \bullet a ) \equiv ( b \bullet ( a \bullet a ) ) \bullet a$
  \\
  right associator-derivation identity & $( b, a \bullet a, c ) \equiv 2 ( b, a, c ) \bullet a$
  \end{tabular}
  \end{center}
A (right) \textbf{pre-Jordan algebra} is a vector space with a bilinear product $a \bullet b$ satisfying
the following two relations of arity 4 where we write for short $a \cdot b = a \bullet b + b \bullet a$:
  \begin{align*}
  &
  (a \cdot b) \bullet (c \bullet d) + (b \cdot c) \bullet (a \bullet d) + (c \cdot a) \bullet (b \bullet d)
  \\[-1mm]
  &\qquad
  \equiv
  c \bullet [(a \cdot b) \bullet d] + a \bullet [(b \cdot c) \bullet d] + b \bullet [(c \cdot a) \bullet d],
  \\
  &
  a \bullet [b \bullet (c \bullet d)] + c \bullet [b \bullet (a \bullet d)] + [(a \cdot c) \cdot b] \bullet d
  \\[-1mm]
  &\qquad
  \equiv
  c \bullet [(a \cdot b) \bullet d] + a \bullet [(b \cdot c) \bullet d] + b \bullet [(c \cdot a) \bullet d].
  \end{align*}
\end{definition}

\begin{lemma}
\cite{Bremner}, \cite{BM}, \cite{HNB}.
In every diassociative algebra the operation $a \vdash b + b \dashv a$ satisfies the identities
defining Jordan dialgebras.
In every dendriform algebra the operation $a \prec b + b \succ a$ satisfies the identities
defining pre-Jordan algebras.
\end{lemma}


\subsection{Manin black and white products}
\label{sectionManin}

We can interpret the vertical arrows in Table \ref{generalizingoperads} using Manin white and black products
of operads; for details see \cite{Bandiera}, \cite{GinzburgKapranov}, \cite{GK2}, \cite{LV}, \cite{Vallette2}.

Starting with $\SAssoc$ in the center of Table \ref{generalizingoperads}, we compute the white product
$\SDiAss \cong \Perm \,\circ\, \SAssoc$ to move up in column 2.
Similarly, we obtain the Leibniz operad $\Leib \cong \Perm \,\circ\, \Lie$ from the operad $\Lie$,
and the di-Jordan operad $\DiJor \cong \Perm \,\circ\, \Jor$ from the Jordan operad $\Jor$.
Taking the white product with $\Perm$ is sometimes called duplication \cite{GK1}, \cite{PBGN};
this process has also been called the KP algorithm \cite{BFSO}.
Starting again from $\SAssoc$, we move down by computing the black product $\SDend \cong \PreLie \bullet \SAssoc$.
Similarly, we obtain $\PreLie \cong \PreLie \bullet \Lie$ and $\PreJor \cong \PreLie \,\circ\, \Jor$.
Taking the black product with $\PreLie$ is sometimes called computing the disuccessor \cite{BBGN}.

For any finitely generated binary quadratic operads $P$ and $Q$, Koszul duality interchanges Manin
black and white products: we have
$( P \bullet Q )^! \cong P^! \,\circ\, Q^!$.
Therefore, $\DiAss$ and $\Dend$ are Koszul dual:
\[
\SDiAss^! \cong ( \SAssoc \,\circ\, \Perm )^! \cong \SAssoc^! \bullet \Perm^! \cong \SAssoc \bullet \PreLie \cong \SDend.
\]
For the Lie column, the operads $\Leib$ and $\PreLie$ are not a dual pair, but we have
\begin{align*}
&
\Leib \cong \Perm \,\circ\, \Lie = \PreLie^! \,\circ\, \ComAss^! \cong ( \PreLie \bullet \ComAss )^! = \Zinb^!,
\\[-2mm]
&
\PreLie \cong \PreLie \bullet \Lie \cong \Perm^! \bullet \ComAss^! \cong ( \Perm \,\circ\, \ComAss )^! \cong \Perm^!.
\end{align*}
For the Jordan column, \emph{if the Jordan operad had a Koszul dual}, then we would obtain:
\begin{align*}
&
\DiJor \cong \Perm \,\circ\, \Jor = \PreLie^! \,\circ\, (\Jor^!)^! \cong ( \PreLie \bullet \Jor^! )^!,
\\[-2mm]
&
\PreJor \cong \PreLie \bullet \Jor \cong \Perm^! \bullet (\Jor^!)^! \cong ( \Perm \,\circ\, \Jor^! )^!.
\end{align*}
See \S\ref{nonquadratickoszulduality} for one way in which a Koszul dual for the Jordan operad could be defined.

Similar considerations apply to the top and bottom rows of Table \ref{generalizingoperads}.
Going from the middle to the top corresponds to taking the triplicator of the operad \cite{GK1}, \cite{PBGN};
this is the white product with the operad $\ComTriAss$ (Definition \ref{triassdef}):
\[
\TriLie \cong \ComTriAss \,\circ\, \Lie, \;
\STriAss \cong \ComTriAss \,\circ\, \SAssoc, \;
\TriJor \cong \ComTriAss \,\circ\, \Jor.
\]
Going from the middle to the bottom represents taking the trisuccessor of the operad \cite{BBGN}:
the black product with the operad $\PostLie$ (Definition \ref{comtriassdef}).
The operads $\STriAss$ and $\STriDend$ at top and bottom of the associative column are Koszul dual.
For the Lie column,
\begin{align*}
&
\TriLie \cong
\ComTriAss \,\circ\, \Lie =
\PostLie^! \,\circ\, \ComAss^! \cong
( \PostLie \bullet \ComAss )^! \cong
\ComTriDend^!,
\\[-2mm]
&
\PostLie \cong
\PostLie \bullet \Lie \cong
\ComTriAss^! \bullet \ComAss^! \cong
( \ComTriAss \,\circ\, \ComAss )^! \cong
\ComTriAss^!.
\end{align*}


\section{Triassociative and tridendriform algebras}
\label{trisection}

\begin{definition}
\label{triassdef}
\cite{LRtrialgebras}.
A \textbf{triassociative algebra} is a vector space with bilinear operations $\dashv$, $\perp$, $\vdash$
satisfying these relations where ${\ast} \in \{ \dashv, \perp, \vdash \}$:
\begin{itemize}
\item
left, middle, and right associativity: \,
$( a \ast b ) \ast c \equiv a \ast ( b \ast c )$,
\item
bar identities: \,
$a \dashv ( b \dashv c ) \equiv a \dashv ( b \ast c )$ and
$( a \vdash b ) \vdash c \equiv ( a \ast b ) \vdash c$,
\item
inner associativity: \,
if $( {\ast_1}, {\ast_2} ) \in \{ ( \vdash, \dashv ), ( \vdash, \perp ), ( \perp, \dashv ) \}$ then
$( a {\,\ast_1\,} b ) {\,\ast_2\,} c \equiv a {\,\ast_1\,} ( b {\,\ast_2\,} c )$,
\item
Loday-Ronco relation: \,
$( a \dashv b ) \perp c \equiv a \perp ( b \vdash c )$.
\end{itemize}
A \textbf{tridendriform algebra} is a vector space with bilinear operations $\prec$, $\curlywedge$, $\succ$
satisfying:
\begin{itemize}
\item
middle associativity: \,
$( a \curlywedge b ) \curlywedge c \equiv a \curlywedge ( b \curlywedge c )$,
\item
left to right and right to left expansions: \,
\item[]
$( a \prec b ) \prec c \equiv a \prec ( b \prec c ) + a \prec ( b \curlywedge c ) + a \prec ( b \succ c )$,
\item[]
$a \succ ( b \succ c ) \equiv ( a \succ b ) \succ c + ( a \curlywedge b ) \succ c + ( a \prec b ) \succ c$,
\item
inner associativity: \,
if $( {\ast_1}, {\ast_2} ) \in \{ ( \succ, \prec ), ( \succ, \curlywedge ), ( \curlywedge, \prec ) \}$ then
$( a {\,\ast_1\,} b ) {\,\ast_2\,} c \equiv a {\,\ast_1\,} ( b {\,\ast_2\,} c )$,
\item
Loday-Ronco relation: \,
$( a \prec b ) \curlywedge c \equiv a \curlywedge ( b \succ c )$.
\end{itemize}
\end{definition}

\begin{definition}
We write $\myfont{BBB}$ for the free nonsymmetric operad generated by three binary operations;
the ordered set $\Omega$ of operations will be either
$\{ \dashv, \perp, \vdash \}$ or $\{ \prec, \curlywedge, \succ \}$.
\end{definition}

\begin{lemma} \label{TATDduality}
\cite{LRtrialgebras}.
The operation $a \cdot b = a \prec b + a \curlywedge b + a \succ b$ is associative in every tridendriform algebra.
The operads $\TriAss$ and $\TriDend$ are a Koszul dual pair.
\end{lemma}

\begin{proof}
(of the second claim).
Both operads are quadratic, binary (and nonsymmetric), so the claim can be verified following
\cite{Loday2001}.
An ordered basis for $\myfont{BBB}(3)$ consists of the following 18 monomials where the pairs
$( \ast_1, \ast_2 )$ are in lex order:
\begin{equation}
\label{18monomials}
( a \ast_1 b ) \ast_2 c, \qquad a \ast_1 ( b \ast_2 c ), \qquad
( \ast_1, \ast_2 ) \in \Omega^2, \qquad \Omega = \{ \dashv, \perp, \vdash \}.
\end{equation}
Let $R$ be the matrix whose rows are the coefficient vectors of the triassociative relations;
thus $\mathrm{rowspace}(R) \subseteq \BBB(3)$.
We write $m_1 \equiv m_2$ as $m_1 - m_2 \equiv 0$ and dot for 0:
\begin{center}
\begin{tabular}{r@{\;}l}
$R =$
&
\scriptsize
$\left[
\begin{array}
{r@{\quad}r@{\quad}r@{\quad}r@{\quad}r@{\;\;}r@{\quad}r@{\quad}r@{\;\;}r@{\;\;}r@{\;\;}r@{\;\;}r@{\;\;}r@{\;\;}r@{\;\;}r@{\;\;}r@{\;\;}r@{\;\;}r}
1 & . & . & . & . & . & . & . & . &-1 & . & . & . & . & . & . & . & . \\
. & . & . & . & 1 & . & . & . & . & . & . & . & . &-1 & . & . & . & . \\
. & . & . & . & . & . & . & . & 1 & . & . & . & . & . & . & . & . &-1 \\
. & . & . & . & . & . & 1 & . & . & . & . & . & . & . & . &-1 & . & . \\
. & . & . & 1 & . & . & . & . & . & . & . & . &-1 & . & . & . & . & . \\
. & . & . & . & . & . & . & 1 & . & . & . & . & . & . & . & . &-1 & . \\
. & . & 1 & . & . &-1 & . & . & . & . & . & . & . & . & . & . & . & . \\
. & . & . & . & . & 1 & . & . &-1 & . & . & . & . & . & . & . & . & . \\
. & . & . & . & . & . & . & . & . & 1 &-1 & . & . & . & . & . & . & . \\
. & . & . & . & . & . & . & . & . & . & 1 &-1 & . & . & . & . & . & . \\
. & 1 & . & . & . & . & . & . & . & . & . & . & . & . &-1 & . & . & .
\end{array}
\right]$
\end{tabular}
\end{center}
We write $R = [ R_1, R_2 ]$ as two blocks of 9 columns, set $R' = [ R_1, -R_2 ]$ to negate
the entries for association type $-(--)$, and compute $\mathrm{RCF}(R')$.
From this we extract the matrix $S$ in RCF whose row space is the nullspace of $R'$:
\begin{center}
\begin{tabular}{r@{\;}l}
$S =$
&
\scriptsize
$\left[
\begin{array}
{r@{\quad}r@{\quad}r@{\quad}r@{\quad}r@{\quad}r@{\quad}r@{\quad}r@{\quad}r@{\;\;}r@{\;\;}r@{\;\;}r@{\;\;}r@{\;\;}r@{\;\;}r@{\;\;}r@{\;\;}r@{\;\;}r}
1 & . & . & . & . & . & . & . & . & -1 & -1 & -1 &  . &  . &  . &  . &  . &  . \\
. & 1 & . & . & . & . & . & . & . &  . &  . &  . &  . &  . & -1 &  . &  . &  . \\
. & . & 1 & . & . & 1 & . & . & 1 &  . &  . &  . &  . &  . &  . &  . &  . & -1 \\
. & . & . & 1 & . & . & . & . & . &  . &  . &  . & -1 &  . &  . &  . &  . &  . \\
. & . & . & . & 1 & . & . & . & . &  . &  . &  . &  . & -1 &  . &  . &  . &  . \\
. & . & . & . & . & . & 1 & . & . &  . &  . &  . &  . &  . &  . & -1 &  . &  . \\
. & . & . & . & . & . & . & 1 & . &  . &  . &  . &  . &  . &  . &  . & -1 &  .
\end{array}
\right]$
\end{tabular}
\end{center}
If we replace the operation symbols by $\{ \prec, \curlywedge, \succ \}$ then the rows of $S$ are the coefficient
vectors of the defining relations for tridendriform algebras.
\end{proof}

\begin{definition}
\label{comtridendef}
\cite{DG},
\cite{Loday2008},
\cite{Ni},
\cite{PBGN}.
A vector space with bilinear operations $\prec$, $\curlywedge$ is a \textbf{commutative tridendriform algebra}
if it satisfies these relations:
\begin{itemize}
\item
middle commutativity and associativity: \,
$a \curlywedge b \equiv b \curlywedge a$, \,
$( a \curlywedge b ) \curlywedge c \equiv a \curlywedge ( b \curlywedge c )$,
\item
middle-left associativity: \,
$( a \curlywedge b ) \prec c \equiv a \curlywedge ( b \prec c )$
\item
left to right expansion: \,
$( a \prec b ) \prec c \equiv a \prec ( b \prec c ) + a \prec ( b \curlywedge c ) + a \prec ( c \prec b )$,
\end{itemize}
A vector space $L$ with bilinear operations $[-,-]$ and $\{-,-\}$ is a (right) \textbf{Lie trialgebra} if:
\begin{itemize}
\item
$( L, [-,-] )$ is a Lie algebra,
$( L, \{-,-\} )$ is a (right) Leibniz algebra,
and
\item
the operations satisfy
$\{ a, [b,c] \} \equiv \{ a, \{b,c\} \}$ and
$\{ [a,b], c \} \equiv [ \{a,c\}, b ] + [ a, \{b,c\} ]$
\end{itemize}
\end{definition}

\begin{lemma}
\label{lietriops}
\cite{DG},
\cite{Ni},
\cite{PBGN}.
In every triassociative algebra the operations $[a,b] = a \perp b - b \perp a$ and
$\{a,b\} = a \dashv b - b \vdash a$ satisfy the identities defining (right) Lie trialgebras.
The operads $\TriLie$ for Lie trialgebras and $\ComTriDend$ for commutative tridendriform algebras
are a Koszul dual pair.
\end{lemma}

\begin{definition}
We write $\BW$ for the free symmetric operad generated by two binary operations, one commutative and
one noncommutative.
In the next proof, we denote these operations by $\curlywedge$ and $\prec$, but later we will use the
symbols $\circ$ and $\bullet$.
\end{definition}

\begin{proof}
(of the second statement in Lemma \ref{lietriops}).
This reduces to a linear algebra calculation as in the proof of Lemma \ref{TATDduality},
but now we have symmetric operads, so we must include all permutations of the arguments.
An ordered basis of $\BW(3)$ consists of 27 monomials:

\vspace{-4mm}

\begin{equation}
\label{BW3basis}
\left\{ \;
\begin{array}{l@{\;\;}l@{\;\;}l@{\;\;}l@{\;\;}l@{\;\;}l}
( a \prec b ) \prec c, &
( a \prec c ) \prec b, &
( b \prec a ) \prec c, &
( b \prec c ) \prec a, &
( c \prec a ) \prec b, &
( c \prec b ) \prec a,
\\
( a \prec b ) \curlywedge c, &
( a \prec c ) \curlywedge b, &
( b \prec a ) \curlywedge c, &
( b \prec c ) \curlywedge a, &
( c \prec a ) \curlywedge b, &
( c \prec b ) \curlywedge a,
\\
( a \curlywedge b ) \prec c, &
( a \curlywedge c ) \prec b, &
( b \curlywedge c ) \prec a, &
( a \curlywedge b ) \curlywedge c, &
( a \curlywedge c ) \curlywedge b, &
( b \curlywedge c ) \curlywedge a,
\\
a \prec ( b \prec c ), &
a \prec ( c \prec b ), &
b \prec ( a \prec c ), &
b \prec ( c \prec a ), &
c \prec ( a \prec b ), &
c \prec ( b \prec a ),
\\
a \prec ( b \curlywedge c ), &
b \prec ( a \curlywedge c ), &
c \prec ( a \curlywedge b ).
\end{array}
\right.
\end{equation}

\vspace{-2mm}

\noindent
We have already accounted for the commutativity of $\curlywedge$.
Each (arity 3) relation in Definition \ref{comtridendef} has six permutations.
Let $R$ be the matrix whose rows are the coefficient vectors of these 18 relations
with respect to the ordered basis \eqref{BW3basis}.
\begin{table}[h!]
\begin{tabular}{r@{\;}l}
$R =$
&
\scriptsize
$\left[
\begin{array}
{
r@{\;}r@{\;}r@{\;}r@{\;}r@{\;}r@{\;}r@{\;}r@{\;}r@{\;}
r@{\;}r@{\;}r@{\;\;}r@{\;\;}r@{\;\;}r@{\;}r@{\;}r@{\;}r@{\;}
r@{\;}r@{\;}r@{\;}r@{\;}r@{\;}r@{\;}r@{\;}r@{\;}r
}
. & . & . & . & . & . & . & . & . & . & . & . & . & . & . & 1 & . & -1 & . & . & . & . & . & . & . & . & . \\
. & . & . & . & . & . & . & . & . & . & . & . & . & . & . & . & 1 & -1 & . & . & . & . & . & . & . & . & . \\
. & . & . & . & . & . & . & . & . & . & . & . & . & . & . & 1 & -1 & . & . & . & . & . & . & . & . & . & . \\
. & . & . & . & . & . & . & . & . & . & . & . & . & . & . & . & -1 & 1 & . & . & . & . & . & . & . & . & . \\
. & . & . & . & . & . & . & . & . & . & . & . & . & . & . & -1 & 1 & . & . & . & . & . & . & . & . & . & . \\
. & . & . & . & . & . & . & . & . & . & . & . & . & . & . & -1 & . & 1 & . & . & . & . & . & . & . & . & . \\
. & . & . & . & . & . & . & . & . & -1 & . & . & 1 & . & . & . & . & . & . & . & . & . & . & . & . & . & . \\
. & . & . & . & . & . & . & . & . & . & . & -1 & . & 1 & . & . & . & . & . & . & . & . & . & . & . & . & . \\
. & . & . & . & . & . & . & -1 & . & . & . & . & 1 & . & . & . & . & . & . & . & . & . & . & . & . & . & . \\
. & . & . & . & . & . & . & . & . & . & -1 & . & . & . & 1 & . & . & . & . & . & . & . & . & . & . & . & . \\
. & . & . & . & . & . & -1 & . & . & . & . & . & . & 1 & . & . & . & . & . & . & . & . & . & . & . & . & . \\
. & . & . & . & . & . & . & . & -1 & . & . & . & . & . & 1 & . & . & . & . & . & . & . & . & . & . & . & . \\
1 & . & . & . & . & . & . & . & . & . & . & . & . & . & . & . & . & . & -1 & -1 & . & . & . & . & -1 & . & . \\
. & 1 & . & . & . & . & . & . & . & . & . & . & . & . & . & . & . & . & -1 & -1 & . & . & . & . & -1 & . & . \\
. & . & 1 & . & . & . & . & . & . & . & . & . & . & . & . & . & . & . & . & . & -1 & -1 & . & . & . & -1 & . \\
. & . & . & 1 & . & . & . & . & . & . & . & . & . & . & . & . & . & . & . & . & -1 & -1 & . & . & . & -1 & . \\
. & . & . & . & 1 & . & . & . & . & . & . & . & . & . & . & . & . & . & . & . & . & . & -1 & -1 & . & . & -1 \\
. & . & . & . & . & 1 & . & . & . & . & . & . & . & . & . & . & . & . & . & . & . & . & -1 & -1 & . & . & -1
\end{array}
\right]$
\end{tabular}
\end{table}

\noindent
The row space of $R$ is the $S_3$-submodule of $\BW(3)$ generated by the commutative tridendriform relations.
For $1 \le i \le 27$ let $\sigma_i \in S_3$ be the permutation of the arguments in the $i$-th monomial
\eqref{BW3basis}.
Let $D$ be the diagonal matrix whose $(i,i)$ entry is $\mathrm{sgn}(\sigma_i)$ if $1 \le i \le 18$
(association type 1) or $-\mathrm{sgn}(\sigma_i)$ if $19 \le i \le 27$ (association type 2).
(This generalizes \cite{Loday2001} to an operad with a commutative operation.)
We set $R' = RD$, compute $\mathrm{RCF}(R')$, and obtain the matrix $S$ in RCF whose row space is the
nullspace of $R'$:

\begin{table}[h!]
\begin{tabular}{r@{\;}l}
$S =$
&
\scriptsize
$\left[
\begin{array}
{
r@{\;}r@{\;}r@{\;}r@{\;}r@{\;}r@{\;}r@{\;}r@{\;}r@{\;}
r@{\;}r@{\;}r@{\;\;}r@{\;\;}r@{\;\;}r@{\;}r@{\;}r@{\;}r@{\;}
r@{\;}r@{\;}r@{\;}r@{\;}r@{\;}r@{\;}r@{\;}r@{\;}r
}
1 & -1 & . & . & . & . & . & . & . & . & . & . & . & . & . & . & . & . & . & . & . & . & . & . & -1 & . & . \\
. & . & 1 & -1 & . & . & . & . & . & . & . & . & . & . & . & . & . & . & . & . & . & . & . & . & . & -1 & . \\
. & . & . & . & 1 & -1 & . & . & . & . & . & . & . & . & . & . & . & . & . & . & . & . & . & . & . & . & -1 \\
. & . & . & . & . & . & 1 & . & . & . & . & -1 & . & -1 & . & . & . & . & . & . & . & . & . & . & . & . & . \\
. & . & . & . & . & . & . & 1 & . & -1 & . & . & -1 & . & . & . & . & . & . & . & . & . & . & . & . & . & . \\
. & . & . & . & . & . & . & . & 1 & . & -1 & . & . & . & -1 & . & . & . & . & . & . & . & . & . & . & . & . \\
. & . & . & . & . & . & . & . & . & . & . & . & . & . & . & 1 & -1 & 1 & . & . & . & . & . & . & . & . & . \\
. & . & . & . & . & . & . & . & . & . & . & . & . & . & . & . & . & . & 1 & . & . & . & . & . & -1 & . & . \\
. & . & . & . & . & . & . & . & . & . & . & . & . & . & . & . & . & . & . & 1 & . & . & . & . & 1 & . & . \\
. & . & . & . & . & . & . & . & . & . & . & . & . & . & . & . & . & . & . & . & 1 & . & . & . & . & -1 & . \\
. & . & . & . & . & . & . & . & . & . & . & . & . & . & . & . & . & . & . & . & . & 1 & . & . & . & 1 & . \\
. & . & . & . & . & . & . & . & . & . & . & . & . & . & . & . & . & . & . & . & . & . & 1 & . & . & . & -1 \\
. & . & . & . & . & . & . & . & . & . & . & . & . & . & . & . & . & . & . & . & . & . & . & 1 & . & . & 1
\end{array}
\right]$
\end{tabular}
\end{table}

\noindent
The rows of $S$ are the coefficient vectors of the relations for Lie trialgebras in Definition \ref{comtridendef}.
We replace the symbols $\curlywedge$, $\prec$ in \eqref{BW3basis} by $[-,-]$, $\{-,-\}$;
since $\curlywedge$ is commutative, its dual $[-,-]$ is anticommutative.
The last 6 rows of $S$ are permutations of $\{ a, [b,c] \} \equiv \{ a, \{b,c\} \}$.
Row 7 is the Jacobi identity for $[-,-]$.
Rows 4--6 are permutations of the statement that $\{-,-\}$ is a derivation of $[-,-]$, namely
$\{ [a,b], c \} \equiv [ \{a,c\}, b ] + [ a, \{b,c\} ]$.
Applying $\{ a, [b,c] \} \equiv \{ a, \{b,c\} \}$ to rows 1--3 gives permutations of the Leibniz identity
for $\{-,-\}$.
\end{proof}

\begin{conjecture}
The operations in Lemma \ref{lietriops} define a functor $F\colon \mathbf{TriAss} \to \mathbf{TriLie}$.
We believe that $F$ has a left adjoint $U$ which sends a Lie trialgebra $L$ to its universal enveloping
triassociative algebra $U(L)$, and that the natural map from $L$ to $U(L)$ is injective.
\end{conjecture}

\begin{definition}
\label{comtriassdef}
\cite{Loday2008}, \cite{Vallette1}.
A vector space $C$ with bilinear operations $\dashv$, $\perp$ is a \textbf{commutative triassociative algebra}
if it satisfies these conditions:
\begin{itemize}
\item
$(C,\perp)$ is a commutative associative algebra, and
$(C,\dashv)$ is a right perm algebra,
\item
inner associativity: \,
$( a \perp b ) \dashv c \equiv a \perp ( b \dashv c )$,
\item
right bar identity: \,
$a \dashv ( b \perp c ) \equiv a \dashv ( b \dashv c )$.
\end{itemize}
A vector space $L$ with bilinear operations $[-,-]$ and $\{-,-\}$ is a (right) \textbf{post-Lie algebra}
if it satisfies these conditions:
\begin{itemize}
\item
$(L,[-,-])$ is a Lie algebra,
\item
$\{-,-\}$ is a right derivation of $[-,-]$, that is,
$\{ [a,b], c \} \equiv [ \{a,c\}, b ] + [ a, \{b,c\} ]$,
\item
the relation \,
$[ \{ \{ a, b \}, c \} - \{ a, \{ b, c \} \} ] - \big[ \{ \{ a, c \}, b \} - \{ a, \{ c, b \} \} \big]
\equiv
\{ a, [b,c] \}$
\end{itemize}
\end{definition}

\begin{lemma}
\label{postlieops}
\cite{Vallette1}.
In every tridendriform algebra the operations $[a,b] = a \curlywedge b - b \curlywedge a$ and
$\{a,b\} = a \prec b - b \succ a$ satisfy the identities defining (right) post-Lie algebras.
The operads $\PostLie$ for post-Lie algebras and $\ComTriAss$ for commutative triassociative algebras
are a Koszul dual pair.
\end{lemma}

\begin{conjecture}
The operations in Lemma \ref{postlieops} define a functor $F\colon \mathbf{TriDend} \to \mathbf{PostLie}$.
We believe that $F$ has a left adjoint $U$ which sends a post-Lie algebra $L$ to its universal enveloping
tridendriform algebra $U(L)$, and that the natural map from $L$ to $U(L)$ is injective.
\end{conjecture}


\section{Jordan trialgebras}
\label{jordantrialgebrasection}

In this section we present the first original contribution of this paper:
we use computer algebra to determine a set of $S_n$-module generators for the multilinear polynomial identities
of arities $n = 3, 4$ satisfied by the Jordan product $a \,\circ\, b = a \perp b + b \perp a$ (commutative) and
Jordan diproduct $a \bullet b = a \dashv b  + b \vdash a$ (noncommutative) in every triassociative algebra.
These generators are the identities defining Jordan trialgebras.
We use representation theory of the symmetric group to show that there are no new identities in arities 5 and 6.

\begin{definition}
We also use $\circ, \bullet$ for the operations in $\BW$; this abuse of notation should not cause confusion.
The \textbf{expansion map} $E(n)\colon \BW(n) \to \STriAss(n)$ is defined on basis monomials in $\BW(n)$ as
follows: $E(n)$ is the identity on arguments, and expands every operation symbol $\circ, \bullet$ using the expressions in the previous paragraph, to produce a linear combination of basis monomials in $\STriAss(n)$
with some permutation of the arguments.
The (nonzero) elements of the kernel of $E(n)$ are the (multilinear) polynomial identities of arity $n$
satisfied by the operations $\circ, \bullet$ in every triassociative algebra.
\end{definition}


\subsection{Relations of arity 3}

\begin{lemma} \label{jordantrialgebralemma}
Over a field $\mathbb{F}$ of characteristic 0 or $p > 3$, every element of the kernel of $E(3)$
is a consequence of the commutativity of $\circ$ and the relation
$a \bullet (b \,\circ\, c) \equiv a \bullet (b \bullet c)$.
\end{lemma}

\begin{definition}
The relation of arity 3 in Lemma \ref{jordantrialgebralemma} is the \textbf{black right bar identity}.
\end{definition}

\begin{proof}
(of Lemma \ref{jordantrialgebralemma}).
We first find monomial bases for $\BW(3)$ and $\STriAss(3)$ so that we can represent $E(3)$ by a matrix
and compute a basis for its nullspace.

Since $\TriAss$ is defined by monomial relations, we find a basis for $\TriAss(3)$ as follows.
The triassociative relations generate an equivalence relation $\sim$ on the set $B$ of 18 (nonsymmetric)
monomials \eqref{18monomials}: we define $t_1 \sim t_2$ if and only if $t_1 \equiv t_2$ is one of the
relations in Definition \ref{triassdef}.
We take the quotient of $B$ by $\sim$ to obtain a set partition $B/{\sim}$.
In each equivalence class we choose, as the normal form of the elements in the class,
the monomial which is minimal with respect to the total order in the proof of Lemma \ref{TATDduality};
we use these representatives as a basis of $\TriAss(3)$.
We obtain the following normal forms; the first two classes each contain four monomials, and the last five
each contain two:
\[
(a {\,\vdash\,} b) {\,\vdash\,} c, \;
(a {\,\dashv\,} b) {\,\dashv\,} c, \;
(a {\,\perp\,}  b) {\,\perp\,}  c, \;
(a {\,\vdash\,} b) {\,\dashv\,} c, \;
(a {\,\perp\,}  b) {\,\dashv\,} c, \;
(a {\,\vdash\,} b) {\,\perp\,}  c, \;
(a {\,\dashv\,} b) {\,\perp\,}  c.
\]
To symmetrize, we apply all 6 permutations of $a, b, c$ to the arguments of the normal forms, and obtain
a basis of 42 monomials for $\STriAss(3)$.

Since $\BW$ is a symmetric operad, to find an ordered basis for $\BW(3)$, we take the 27 monomials \eqref{BW3basis}
from the proof of Lemma \ref{lietriops} and replace the symbols $\prec, \curlywedge$ by $\bullet, \circ$.

The columns of the $42 \times 27$ matrix representing $E(3)$ with these ordered bases are
the coefficient vectors of the expansions of the basis monomials of $\BW(3)$ into $\STriAss(3)$;
we must replace the triassociative monomials by their normal forms.
For example,
\begin{alignat*}{2}
(a \,\circ\, b) \bullet c
&\mapsto
( a \perp b  + b \perp a ) \dashv c
+
c \vdash ( a \perp b  + b \perp a )
&\qquad
&\text{expansion map}
\\[-2pt]
&=
( a \perp b ) \dashv c +
( b \perp a ) \dashv c +
c \vdash ( a \perp b ) +
c \vdash ( b \perp a )
&\qquad
&\text{bilinearity}
\\[-2pt]
&=
( a \perp b ) \dashv c + ( b \perp a ) \dashv c
+
( c \vdash a ) \perp b  + ( c \vdash b ) \perp a
&\qquad
&\text{normal form}
\end{alignat*}
The matrix has rank 21 and nullity 6.
We find a basis for the nullspace and write down the corresponding identities,
which are the 6 permutations of the black right bar identity.
\end{proof}

\begin{remark}
\label{rightbarremark}
The black right bar identity and the commutativity of $\circ$ imply right commutativity for $\bullet$
(Definition \ref{jordialgdef}): we have
$a \bullet ( b \bullet c ) \equiv a \bullet ( b \circ c ) \equiv a \bullet ( c \circ b ) \equiv a \bullet ( c \bullet b )$.
\end{remark}


\subsection{The operad $\BW$: skeletons, total order, normal forms}

Before studying arities $n \ge 4$, we need to understand skeletons in $\BW$;
as before, a skeleton in arity $n$ is a placement of parentheses and a choice of operation symbols
in a sequence of $n$ indistinguishable arguments.
For the operad $\BW$, their enumeration is not trivial due to the commutativity of $\circ$.

\begin{definition}
\label{BWSdefinition}
It is convenient to consider $\BW$ not in the monoidal category $\VectF$ with tensor product,
but in $\mathbf{Set}$ with direct product; we call the latter operad
$\BW$-$\mathbf{Set}$.
(Equivalently, we consider a basis of each $\BW(n)$ rather than the entire vector space.)
Since $\BW$-$\mathbf{Set}$ is a symmetric operad, we cannot use the symbol $x$ for every argument.
But we achieve the same goal by considering the free algebra $\BWS$ over $\BW$-$\mathbf{Set}$ generated
by $x$; this algebra consists of all $\BW$-\textbf{skeletons}.
We write $\BWS(n)$ for the subset of $\BWS$ consisting of the elements of arity $n$;
that is, those in which $x$ occurs exactly $n$ times.
This device allows us to use the symbol $x$ as a placeholder for unspecified arguments in a symmetric operad.
\end{definition}

\begin{definition}
\label{BWorderdefinition}
For $n \ge 1$, we inductively define a \textbf{total order} $\ll_n$ on $\BWS(n)$:
\begin{enumerate}[(1)]
\item
For $n = 1$, we have $\BWS(1) = \{ x \}$.
In what follows we omit the subscript on $\ll_n$.
\item
For $n \ge 2$, if $h \in \BWS(n)$ then $h = f \,\circ\, g$ or $h = f \bullet g$ with
$1 \le \mathrm{arity}(f), \mathrm{arity}(g) < n$.
\begin{enumerate}[(a)]
\item
We set $f_1 \,\circ\, g_1 \ll f_2 \bullet g_2$ for all $f_1 \,\circ\, g_1, \, f_2 \bullet g_2 \in \BWS(n)$.
\item
Consider $f_1 \bullet g_1$ and $f_2 \bullet g_2$.
If $\mathrm{arity}(f_1) > \mathrm{arity}(f_2)$, or
$\mathrm{arity}(f_1) = \mathrm{arity}(f_2)$ and $f_1 \ll f_2$, or
$f_1 = f_2$ and $g_1 \ll g_2$, then we set $f_1 \bullet g_1 \ll f_2 \bullet g_2$.
\item
Consider  $f_1 \circ g_1$ and $f_2 \circ g_2$.
We may assume $\mathrm{arity}(f_i) \ge \mathrm{arity}(g_i)$ for $i = 1, 2$.
\begin{enumerate}[(i)]
\item
If $\mathrm{arity}(f_1) > \mathrm{arity}(f_2)$ then we set $f_1 \,\circ\, g_1 \ll f_2 \,\circ\, g_2$.
\item
If $\mathrm{arity}(f_1) = \mathrm{arity}(f_2)$ then we also have $\mathrm{arity}(g_1) = \mathrm{arity}(g_2)$.
\begin{itemize}
\item
Assume $\mathrm{arity}(f_1) > \mathrm{arity}(g_1)$.
If $f_1 \ll f_2$, or $f_1 = f_2$ and $g_1 \ll g_2$, then we set $f_1 \circ g_1 \ll f_2 \circ g_2$.
\item
Assume $\mathrm{arity}(f_1) = \mathrm{arity}(g_1)$.
In this case, $n$ is even and all four factors have arity $n/2$.
By commutativity, $f_i \ll g_i$ (or $f_i = g_i$) for $i = 1, 2$.
If $f_1 \ll f_2$, or $f_1 = f_2$ and $g_1 \ll g_2$, then we set $f_1 \circ g_1 \ll f_2 \circ g_2$.
\end{itemize}
\end{enumerate}
\end{enumerate}
\end{enumerate}
\end{definition}

\begin{example}
\label{25skeletons}
The 25 skeletons in $\BWS(4)$ in the order of Definition \ref{BWorderdefinition} are as follows:
\vspace{-1mm}
\[
\begin{array}{l@{\;\;}l@{\;\;}l@{\;\;}l@{\;\;}l}
( ( x \circ x ) \circ x ) \circ x, &
( ( x \bullet x ) \circ x ) \circ x, &
( ( x \circ x ) \bullet x ) \circ x, &
( ( x \bullet x ) \bullet x ) \circ x, &
( x \bullet ( x \circ x ) ) \circ x, \\[-1pt]
( x \bullet ( x \bullet x ) ) \circ x, &
( x \circ x ) \circ ( x \circ x ), &
( x \circ x ) \circ ( x \bullet x ), &
( x \bullet x ) \circ ( x \bullet x ), &
( ( x \circ x ) \circ x ) \bullet x, \\[-1pt]
( ( x \bullet x ) \circ x ) \bullet x, &
( ( x \circ x ) \bullet x ) \bullet x, &
( ( x \bullet x ) \bullet x ) \bullet x, &
( x \bullet ( x \circ x ) ) \bullet x, &
( x \bullet ( x \bullet x ) ) \bullet x, \\[-1pt]
( x \circ x ) \bullet ( x \circ x ), &
( x \circ x ) \bullet ( x \bullet x ), &
( x \bullet x ) \bullet ( x \circ x ), &
( x \bullet x ) \bullet ( x \bullet x ), &
x \bullet ( ( x \circ x ) \circ x ), \\[-1pt]
x \bullet ( ( x \bullet x ) \circ x ), &
x \bullet ( ( x \circ x ) \bullet x ), &
x \bullet ( ( x \bullet x ) \bullet x ), &
x \bullet ( x \bullet ( x \circ x ) ), &
x \bullet ( x \bullet ( x \bullet x ) ).
\end{array}
\]
\end{example}

\subsubsection{Computing the normal form of a multilinear $\BW$ monomial}

We apply commutativity of $\circ$ to straighten first the skeleton and then the arguments; for example,
$c \circ ( b \circ a )$ becomes $( b \circ a ) \circ c$ and then $( a \circ b ) \circ c$.
Straightening the skeleton is equivalent to choosing a unique representative of each equivalence
class in $\BWS(n)$ with respect to commutativity:
the skeleton $x \circ ( x \circ x )$ has representative $( x \circ x ) \circ  x$.
This algorithm is necessary when we compute the consequences in arity $n{+}1$ of a $\BW$ polynomial
in arity $n$: applying the substitution maps of Definition \ref{composeskeletons} to the set of representatives of equivalence classes in $\BWS(n)$ does not necessarily produce representatives of equivalence classes in $\BWS(n+1)$.

\begin{algorithm}
For $n = 1$ there is only one multilinear monomial $x_1$ which by definition is already in normal form.
For $n \ge 2$, every multilinear monomial of arity $n$ has the form $f \,\circ\, g$ or $ f \bullet g$
(we write $f \ast g$ to cover both cases) where $1 \le \mathrm{arity}(f), \mathrm{arity}(g) < n$:
\begin{enumerate}[(1)]
\item
Recursively compute $f'$ and $g'$, the normal forms of $f$ and $g$.
\item
If $\mathrm{arity}(f') > \mathrm{arity}(g')$ then return $f' \ast g'$.
\item
If $\mathrm{arity}(f') < \mathrm{arity}(g')$ then:
if $\ast \; = \circ$ then return $g' \ast f'$ else return $f' \ast g'$.
\item
If $\mathrm{arity}(f') = \mathrm{arity}(g')$ then:
\begin{enumerate}[(a)]
\item
Extract $s(f')$ and $s(g')$, the skeletons of $f'$ and $g'$.
\item
If $s(f') \ll s(g')$ in the total order on $\BWS(n/2)$ then return $f' \ast g'$.
\item
If $s(f') \gg s(g')$ then:
if $\ast \; = \circ$ then return $g' \ast f'$ else return $f' \ast g'$.
\item
If $s(g') = s(f')$ then:
\begin{enumerate}[(i)]
\item
Extract $p(f')$ and $p(g')$, the sequences of subscripts of the arguments of $f'$ and $g'$;
as sets, $p(f') \cap p(g') = \emptyset$ and $p(f') \cup p(g') = \{1,\dots,n\}$.
\item
If $p(f')$ precedes $p(g')$ in lex order then return $f' \ast g'$, where lex order means to compare
the leftmost unequal elements.
\item
If $p(g')$ precedes $p(f')$ in lex order then:
if $\ast \; = \circ$ then return $g' \ast f'$ else return $f' \ast g'$.
\item
If $p(f') = p(g')$ then return $f' \ast g'$.
\end{enumerate}
\end{enumerate}
\end{enumerate}
\end{algorithm}


\subsection{Relations of arity 4}

\begin{theorem}
\label{jordantrialgebraproposition}
Over a field $\mathbb{F}$ of characteristic 0 or $p > 4$, every multilinear polynomial identity of
arity $\le 4$, satisfied by the Jordan product and diproduct in every
triassociative algebra, is a consequence of the commutativity of $\circ$, the black right
bar identity, and the (linearizations of the) following identities of arity 4:
\begin{align}
\label{jordan}
&
( (a \,\circ\, a) \,\circ\, b ) \,\circ\, a
\equiv
( a \,\circ\, a ) \,\circ\, ( b \,\circ\, a ),
\\[-1mm]
\label{quasijordan1}
&
( a \bullet (b \bullet b ) ) \bullet b
\equiv
( a \bullet b ) \bullet ( b \bullet b ),
\\[-1mm]
\label{quasijordan2}
&
( ( a \bullet b ) \bullet d ) \bullet c
+
( ( a \bullet c ) \bullet d ) \bullet b
+
a \bullet ( ( b \bullet c ) \bullet d )
\\[-1mm]
\notag
&\qquad
\equiv
( a \bullet ( b \bullet c ) ) \bullet d
+
( a \bullet ( b \bullet d ) ) \bullet b
+
( a \bullet ( c \bullet d ) ) \bullet b,
\\[-1mm]
\label{JTAnew1}
&
  ( a  \,\circ\,  a )  \,\circ\,  ( a \bullet b )
  \equiv
  ( ( a  \,\circ\,  a ) \bullet b )  \,\circ\,  a,
\\[-1mm]
\label{JTAnew2}
&
    ( ( a \bullet d )  \,\circ\,  c )  \,\circ\,  b
  + ( ( b \bullet d )  \,\circ\,  c )  \,\circ\,  a
  + ( ( a  \,\circ\,  b )  \,\circ\,  c ) \bullet d
\\[-1mm]
\notag
&\qquad \equiv
    ( ( a  \,\circ\,  b ) \bullet d )  \,\circ\,  c
  + ( ( a  \,\circ\,  c ) \bullet d )  \,\circ\,  b
  + ( ( b  \,\circ\,  c ) \bullet d )  \,\circ\,  a,
\\[-1mm]
\label{JTAnew3}
&
    ( ( a \bullet d ) \bullet c )  \,\circ\,  b
  + ( ( b \bullet d ) \bullet c )  \,\circ\,  a
  + ( ( a  \,\circ\,  b ) \bullet c ) \bullet d
\\[-1mm]
\notag
&\qquad \equiv
    ( b \bullet ( c \bullet d ) )  \,\circ\,  a
  + ( ( a \bullet c )  \,\circ\,  b ) \bullet d
  + ( ( a \bullet d )  \,\circ\,  b ) \bullet c,
\\[-1mm]
\label{JTAnew4}
&
    ( a \bullet c )  \,\circ\,  ( b \bullet d )
  + ( a \bullet d )  \,\circ\,  ( b \bullet c )
  + ( a  \,\circ\,  b ) \bullet ( c \bullet d )
\\[-1mm]
\notag
&\qquad \equiv
    ( b \bullet ( c \bullet d ) )  \,\circ\,  a
  + ( ( a \bullet c )  \,\circ\,  b ) \bullet d
  + ( ( a \bullet d )  \,\circ\,  b ) \bullet c.
\end{align}
\end{theorem}

\begin{proof}
Similar to that of Lemma \ref{jordantrialgebralemma}, but the matrices are larger,
and we must deal with the consequences of the black right bar identity of arity 3.
We first construct the 135 basis skeletons in $\BBB(4)$, where $\ast_1, \ast_2, \ast_3$
are chosen freely from $\{ \dashv, \perp, \vdash \}$:
\[
( ( x { \ast_1 } x ) { \ast_2 } x ) { \ast_3 } x, \;
( x { \ast_1 } ( x { \ast_2 } x ) ) { \ast_3 } x, \;
( x { \ast_1 } x ) { \ast_2 } ( x { \ast_3 } x ), \;
x { \ast_1 } ( ( x { \ast_2 } x ) { \ast_3 } x ), \;
x { \ast_1 } ( x { \ast_2 } ( x { \ast_3 } x ) ).
\]
Second, we generate the 165 consequences in arity 4 of the defining relations for $\TriAss$
using \eqref{operadcompositions}: each relation $f(a,b,c)$ has 15 consequences where $\ast$ is
chosen from $\{ \dashv, \perp, \vdash \}$:
\[
f( a \ast d, b, c ), \quad\;
f( a, b \ast d, c ), \quad\;
f( a, b, c \ast d ), \quad\;
f( a, b, c ) \ast d, \quad\;
d \ast f( a, b, c ).
\]
These consequences generate an equivalence relation on the basis of $\myfont{BBB}(4)$,
which has 15 classes representing a basis of $\TriAss(4)$.
We choose the minimal element in each class as the normal form of the (nonsymmetric) monomials in that class:
\[
\begin{array}{llll}
( ( a \vdash b ) \vdash c ) \vdash d, &
( ( a \vdash b ) \vdash c ) \dashv d, &
( ( a \vdash b ) \vdash c ) \perp d,  &
( ( a \vdash b ) \dashv c ) \dashv d, \\
( ( a \vdash b ) \dashv c ) \perp d,  &
( ( a \vdash b ) \perp c ) \dashv d,  &
( ( a \vdash b ) \perp c ) \perp d,   &
( ( a \dashv b ) \dashv c ) \dashv d, \\
( ( a \dashv b ) \dashv c ) \perp d,  &
( ( a \dashv b ) \perp c ) \dashv d,  &
( ( a \dashv b ) \perp c ) \perp d,   &
( ( a \perp b ) \dashv c ) \dashv d,  \\
( ( a \perp b ) \dashv c ) \perp d,   &
( ( a \perp b ) \perp c ) \dashv d,   &
( ( a \perp b ) \perp c ) \perp d.
\end{array}
\]
To construct a monomial basis of the symmetrization, we apply all permutations of $a, b, c, d$;
thus $\STriAss(4)$, the codomain of $E(4)$, has dimension 360.

The domain of $E(4)$ is the space $\BW(4)$.
To generate an ordered basis of $\BW(4)$, we start with the 25 skeletons in Example \ref{25skeletons}.
In each skeleton, we replace the four occurrences of the symbol $x$ by all permutations of $x_1,\dots,x_4$;
for each resulting multilinear monomial we compute the normal form with respect to commutativity,
and save only those monomials which are irreducible (equal to their normal forms).
The number of monomials we obtain for each skeleton is $4!/2^s$ where $s$ is the number of commutative symmetries
in the skeleton, where by definition a symmetry is a sub-skeleton of the form $f \circ f$.
For example, the skeleton $( x \,\circ\, x ) \,\circ\, ( x \,\circ\, x )$ has three symmetries:
one with $f = x \circ x$ and two with $f = x$, so it has only three multilinear monomials,
corresponding to permutations 1234, 1324, 1423 of the subscripts.
We order these monomials first by skeleton and then by permutation.
The total number of multilinear monomials over all skeletons is $\dim\BW(4) = 405$.

With respect to the ordered monomial bases of $\BW(4)$ and $\STriAss(4)$, the expansion map
$E(4)\colon \BW(4) \to \STriAss(4)$ is represented by the $360 \times 405$ matrix $E_4$.
The $(i,j)$ entry of $E_4$ is the coefficient of the $i$-th $\STriAss$ monomial in the expansion
of the $j$-th $\BW$ monomial.
We must replace each $\STriAss$ monomial $m$ in each expansion by the equivalent element in the monomial
basis for $\STriAss(4)$; we replace the $\TriAss$ skeleton of $m$ by the representative of the corresponding
equivalence class.
For example,
\[
\begin{array}{l@{\;}l}
&
E(4)\big( \, ( a \bullet b ) \,\circ\, ( c \bullet d ) \,\big)
\\
=
&
( a \dashv b ) \perp ( c \dashv d ) +
( c \dashv d ) \perp ( a \dashv b ) +
( a \dashv b ) \perp ( d \vdash c ) +
( d \vdash c ) \perp ( a \dashv b ) + {}
\\
&
( b \vdash a ) \perp ( c \dashv d ) +
( c \dashv d ) \perp ( b \vdash a ) +
( b \vdash a ) \perp ( d \vdash c ) +
( d \vdash c ) \perp ( b \vdash a )
\\
=
&
( ( a \dashv b ) \perp c ) \dashv d +
( ( c \dashv d ) \perp a ) \dashv b +
( ( a \dashv b ) \dashv d ) \perp c +
( ( d \vdash c ) \perp a ) \dashv b + {}
\\
&
( ( b \vdash a ) \perp c ) \dashv d +
( ( c \dashv d ) \dashv b ) \perp a +
( ( b \vdash a ) \dashv d ) \perp c +
( ( d \vdash c ) \dashv b ) \perp a.
\end{array}
\]
In this way, we initialize $E_4$, which has entries in $\{ 0, 1 \}$.
This matrix is large, so we use modular arithmetic ($p = 101$) to compute its RCF and the corresponding
basis for its nullspace.

We find that $E_4$ has rank 165 and nullity 240.
Its nullspace $N(4) \subseteq \BW(4)$ is the $S_4$-module which contains the coefficient vectors
of the multilinear identities of arity 4 satisfied by the Jordan product and diproduct in every
triassociative algebra.
The new identities are the (nonzero) elements of the quotient module $N(4)/M(4)$ where
$M(4)$ is the $S_4$-module generated by the known identities:
\begin{enumerate}[(1)]
\item
Consequences of the black right bar identity,
$f(a,b,c) = a \bullet (b \,\circ\, c) - a \bullet (b \bullet c) \equiv 0$,
which generate an $S_4$-module of dimension 180:
$f( a \,\circ\, d, b, c )$,
$f( a, b \,\circ\, d, c )$,
$f( a, b, c \,\circ\, d )$,
$f( a, b, c ) \,\circ\, d$,
$f( a \bullet d, b, c )$,
$f( a, b \bullet d, c )$,
$f( a, b, c \bullet d )$,
$f( a, b, c ) \bullet d$,
$d \bullet f( a, b, c )$.
\item
Linearization of the Jordan identity for $a \,\circ\, b = a \perp b + b \perp a$.
Combining this with the previous generators, we obtain an $S_4$-module of dimension 184.
\item
Remark \ref{rightbarremark} shows that we do not need to consider the consequences of
right commutativity for $\bullet$, since they are also consequences
of the black right bar identity.
\item
Linearizations of the Jordan dialgebra identities for
$a \bullet b = a \dashv b + b \vdash a$.
Combining these with the other generators, we obtain the $S_4$-module $M(4)$ of dimension 200.
\end{enumerate}
We represent $M(4)$ as the row space of a $200 \times 405$ matrix $M_4$.
The quotient module $N(4)/M(4)$ of new identities of arity 4 has dimension 40.

From the RCF of $E_4$ we extract a ``canonical'' basis of $N(4)$ as follows:
for each standard basis vector $v$ of $\mathbb{F}_p^{240}$, we set the 240 free variables equal to
the components of $v$ and
solve for the leading variables.
We put the resulting basis vectors into the rows of a $240 \times 405$ matrix $N_4$ whose row space is $N(4)$.
(A canonical basis for $N(4)$ would consist of the rows of $\mathrm{RCF}(N_4)$.)
Checking the rows of $N_4$ (all identities) one-by-one against the row space of $M_4$ (known identities),
we obtain a list of 40 rows which increase the rank.
Each row contains 6 nonzero entries, representing an identity with 6 terms.
The cosets of these row vectors form a basis of the quotient module $N(4)/M(4)$.

From this linear basis of 40 vectors, we extract a minimal set of $S_4$-module generators.
Checking again the rows $\rho$ of $N_4$ against the row space of $M_4$, but now including all
24 permutations of the identity with coefficient vector $\rho$,
we obtain 5 rows which increase the rank to 204, 216, 222, 232, 240.
We find that the third row belongs to the submodule generated by the others.
The remaining four generators are identities \eqref{JTAnew1}-\eqref{JTAnew4}.
\end{proof}

\begin{remark}
The operation $\circ$ satisfies the Jordan identity \eqref{jordan};
the operation $\bullet$ satisfies the Jordan dialgebra identities \eqref{quasijordan1}-\eqref{quasijordan2};
the operations are related by \eqref{JTAnew1}-\eqref{JTAnew4}.
This pattern of identities is similar to Definition \ref{comtridendef} of Lie trialgebras:
$[a,b]$ defines a Lie algebra, $\{a,b\}$ defines a Leibniz algebra, and other identities relate the two operations.
\end{remark}

\begin{definition}
\label{triJordandef}
Over a field $\mathbb{F}$ of characteristic 0 or $p > 4$, a vector space $J$ with bilinear operations
$\circ$ and $\bullet$ is a \textbf{Jordan trialgebra} if:
\begin{itemize}
\item
$( J, \,\circ\, )$ is a Jordan algebra, and $( J, \,\bullet\, )$ is a Jordan dialgebra, and
\item
the two operations are related by identities \eqref{JTAnew1}--\eqref{JTAnew4} of Proposition \ref{jordantrialgebraproposition}.
\end{itemize}
The corresponding symmetric operad will be denoted $\TriJor$.
\end{definition}


\subsection{Triplicators}
\label{triplicatorsection}

The operad $\TriJor$ governing Jordan trialgebras may also be obtained using the techniques of
Pei et al.~\cite{PBGN}.
We start from the linearized Jordan identity,
$((ab)d)c + ((ac)d)b + ((bc)d)a - (ab)(cd) - (ac)(bd) - (ad)(bc) \equiv 0$,
and its representation in terms of tree monomials, using $\cdot$ for the operation symbol:
\vspace{-3mm}
\begin{figure}[h]
\centering
\includegraphics[height=2.3cm]{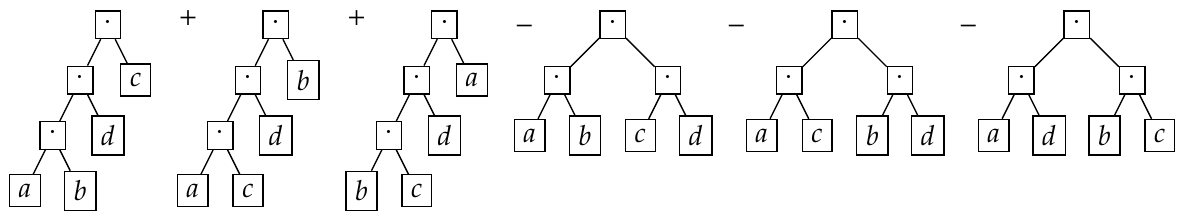}
\vspace{-2mm}
\end{figure}

\vspace{-3mm}

\normalsize

\noindent
Let $L \subseteq \{ a,b,c,d \}$ be a nonempty subset of the arguments.
For each $L$ we construct a new polynomial identity involving operations $\circ_1$, $\circ_2$, $\circ_3$
which we denote simply by 1, 2, 3.
We apply the following triplicator algorithm to each tree monomial $m$ in the identity above.

\begin{algorithm}
\label{tripalgorithm}
For each $x \in L$, let $p_m(x)$ be the unique path from the root of $m$ to the leaf $x$.
Let $W_m$ be the set of internal vertices of $m$, and define
$t_m\colon L \times W_m \to \{ \leftarrow, 0, \rightarrow \}$:
\begin{itemize}
\item
If the internal vertex $v \in W_m$ does not lie on the path $p_m(x)$ then $t_m(x,v) = 0$.
\item
If $v \in W_m$ lies on $p_m(x)$ then $t_m(x,v) = \,\leftarrow$ (resp.~$\rightarrow$) if $p_m(x)$ turns left
(resp.~right) at $v$.
\end{itemize}
There are four possibilities for $T_m(L,v) = \{ \, t_m(x,v) \mid x \in L \, \} \setminus \{ 0 \}$:
\begin{itemize}
\item
If $T_m(L,v) = \{ \leftarrow \}$ then the internal vertex $v$ receives the new operation symbol $1$.
\item
If $T_m(L,v) = \{ \rightarrow \}$ then the internal vertex $v$ receives the new operation symbol $3$.
\item
If $T_m(L,v) = \{ \leftarrow, \rightarrow \}$ then the internal vertex $v$ receives the new operation symbol $2$.
\item
If $T_m(L,v) = \emptyset$ then none of the paths $p_m(x)$ pass through $v$, and in this case $v$ receives
the symbol $\ast$ representing the union of the new operation symbols:
$\ast \,= \{ 1, 2, 3 \}$.
\end{itemize}
\end{algorithm}

Since the linearized Jordan identity is symmetric in $a, b, c$ it suffices to consider only these subsets $L$:
$\{ a \}$, $\{ d \}$, $\{ a,b \}$, $\{ a,d \}$, $\{ a,b,c \}$, $\{ a,b,d \}$, $\{ a,b,c,d \}$.
For example, if $L = \{ a,c \}$ then Algorithm \ref{tripalgorithm} produces this new identity:

\vspace{-3mm}
\begin{figure}[h]
\centering
\includegraphics[scale=1]{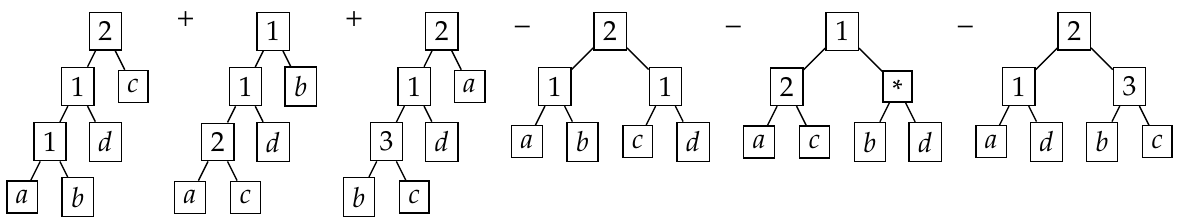}
\vspace{-2mm}
\end{figure}

\normalsize

\noindent
Note that term 5 contains a vertex labelled $\ast$; it thus produces three new identities
after we replace $\ast$ by $1$, then $2$, then $3$.
These new operations satisfy symmetries which follow from the commutativity of the Jordan product:
$a \,\circ_1\, b \equiv b \,\circ_3\, a$ and $a \,\circ_2\, b \equiv b \,\circ_2\, a$.
We replace $\ast$ in term 5 by operation $1$, and apply the symmetries of
the operations to eliminate operation $3$ from terms 3 and 6.
The final result involves only two new operations $1, 2$:

\vspace{-3mm}
\begin{figure}[h]
\centering
\includegraphics[scale=1]{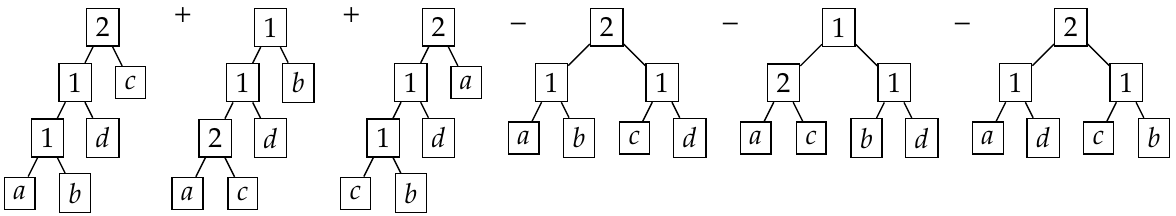}
\vspace{-2mm}
\end{figure}

\normalsize

\noindent
If we replace 2 by $\circ$, and  1 by $\bullet$, and rewrite the tree monomials in parenthesized form,
then we obtain multilinear identities which we can compare directly to those of
Theorem \ref{jordantrialgebraproposition}.

\begin{definition}
The operad defined by the binary operations $\circ$ and $\bullet$, where $\circ$ is commutative and $\bullet$ has no symmetry,
satisfying the multilinear relations resulting from all possible applications of Algorithm \ref{tripalgorithm} to the linearized Jordan identity,
is denoted $\myfont{TripJor}$ and is called the \textbf{triplicator} of the Jordan operad $\Jor$.
\end{definition}

\begin{proposition}
Let $\TriJor$ be the symmetric operad governing Jordan trialgebras,
generated by the operations $\circ$ (commutative) and $\bullet$ (no symmetry)
satisfying the identities of Theorem \ref{jordantrialgebraproposition}.
Let $\myfont{TripJor}$ be the symmetric operad generated by the same operations but
satisfying the identities obtained by applying the triplicator Algorithm \ref{tripalgorithm}
to the linearized Jordan identity.
These two sets of identities generate the same subquotient of the $S_4$-module $\BW(4)$,
and therefore the two operads are isomorphic: $\TriJor \cong \myfont{TripJor}$.
\end{proposition}


\subsection{Representation theory of the symmetric group}

We can extend these calculations to arities $n \ge 5$, but the matrices become very large, and so
we use the representation theory of the symmetric group to decompose the $S_n$-modules into
isotypic components.
For the theoretical background see the survey \cite{BMP}, which includes an extensive bibliography.
We mention Rutherford's exposition \cite{Rutherford} of Young's work,
Hentzel's implementation \cite{Hentzel} on a computer,
and the important contributions by his students Bondari \cite{Bondari} and Clifton \cite{Clifton}.

For $n \ge 1$, over a field $\mathbb{F}$ of characteristic 0 or $p > n$, the group algebra $\mathbb{F}S_n$
is semisimple and hence decomposes as the direct sum of simple two-sided ideals isomorphic to full matrix rings.
These ideals are in bijection with the partitions of $n$: for partition $\lambda$ we have the matrix
ring $M_{d_\lambda}(\mathbb{F})$ where the dimension $d_\lambda$ is given by the hook formula.
If $Y(\lambda)$ is the Young diagram for $\lambda$ then $d_\lambda$ is the number of standard tableaux
for $Y(\lambda)$.
For each $\lambda$ the projection $R_\lambda\colon \mathbb{F}S_n \to M_{d_\lambda}(\mathbb{F})$ defines
an irreducible representation of $S_n$; we call $R_\lambda(\pi)$ the representation matrix of $\pi$ for
partition $\lambda$.
If we define the action of $S_n$ on $M_{d_\lambda}(\mathbb{F})$ by $\pi \cdot A = R_\lambda(\pi) A$
for $A \in M_{d_\lambda}(\mathbb{F})$, then $M_{d_\lambda}(\mathbb{F})$ becomes a (left) $S_n$-module,
which is the direct sum of $d_\lambda$ isomorphic minimal left ideals in $\mathbb{F}S_n$:
the column vectors in $M_{d_\lambda}(\mathbb{F})$.
We write $[\lambda]$ for the isomorphism class of these minimal left ideals; $[\lambda]$ is the irreducible
$S_n$-module for partition $\lambda$.
Clifton \cite{Clifton} found an efficient algorithm to compute $R_\lambda(\pi)$.


\subsubsection{All relations: kernel of the expansion map}

We show how to compute the kernel of the expansion map $E(n)$ using representation theory;
recall that $E(n)$ is a linear map from $\BW(n)$ to $\STriAss(n)$.
Let $\BWS(n) = \{ \beta_1, \, \dots, \, \beta_{b(n)} \}$ be the ordered set of skeletons
for $\BW(n)$ (Definition \ref{BWSdefinition}),
and let $\myfont{NTA}(n) = \{ \tau_1, \, \dots, \, \tau_{t(n)} \}$ be the ordered basis
for $\TriAss(n)$ -- the representatives of the equivalence classes in $\BBB(n)$ -- which are
the skeletons for $\STriAss(n)$.
The advantage of representation theory is that we do not need all multilinear monomials,
but only the skeletons: a much smaller set.
For each partition $\lambda$ of $n$ we compute the matrix $E_{n,\lambda}$,
which has $b(n)$ rows and $t(n)+b(n)$ columns of $d_\lambda \times d_\lambda$ blocks.
\begin{equation}
\label{Enlambda}
E_{n,\lambda}
=
\left[
\begin{array}{cccc|cccc}
R_\lambda( X_{1,1} ) & R_\lambda( X_{1,2} ) & \cdots & R_\lambda( X_{1,t(n)} ) &
I_{d_\lambda} & 0 & \cdots & 0
\\[1mm]
R_\lambda( X_{2,1} ) & R_\lambda( X_{2,2} ) & \cdots & R_\lambda( X_{2,t(n)} ) &
0 & I_{d_\lambda} & \cdots & 0
\\[1mm]
\vdots & \vdots & \ddots & \vdots & \vdots & \vdots & \ddots & \vdots
\\[1mm]
R_\lambda( X_{b(n),1} ) & R_\lambda( X_{b(n),2} ) & \cdots & R_\lambda( X_{b(n),t(n)} ) &
0 & 0 & \cdots & I_{d_\lambda}
\end{array}
\right]
\end{equation}
For $1 \le i \le b(n)$ and $1 \le j \le t(n)$, position $(i,j)$ contains the representation matrix
$R_\lambda( X_{i,j} )$ where $X_{i,j} \in \mathbb{F}S_n$ is determined as follows:
\begin{enumerate}[(i)]
\item
Start with the $i$-th skeleton $\beta_i$ from $\myfont{BWS}(n)$.
\item
Replace the $n$ occurrences of the symbol $x$ in $\beta_i$ by the identity permutation
$x_1, \dots, x_n$ of $n$ indeterminates to obtain the multilinear basis monomial $\mu_i$ for $\BW(n)$.
\item
Apply $E(n)$ to $\mu_i$, obtaining $X_i = E(n)(\mu_i) \in \TriAss(n)$.
The terms of the expansion $E(n)(\mu_i)$ belong to the symmetrization of $\BBB(n)$, so we
replace each $\BBB$ skeleton by the representative of its triassociative equivalence class;
see the proofs of Lemma \ref{jordantrialgebralemma} and Proposition \ref{jordantrialgebraproposition}.
After this normalization of the skeletons, $X_i \in \STriAss(n)$.
\item
Decompose $X_i = X_{i,1} + \cdots + X_{i,t(n)}$ as a sum of $t(n)$ components where $X_{i,j}$ consists
of the terms of $X_i$ with triassociative skeleton $\tau_j$.
\item
In $X_{i,j}$ every term has skeleton $\tau_j$; the only differences are in the coefficients and the permutations.
Thus $X_{i,j}$ is a linear combination of permutations of $n$, so $X_{i,j} \in \mathbb{F}S_n$.
\item
Compute $R_\lambda( X_{i,j} )$, the representation matrix of $X_{i,j}$ and store it in block $(i,j)$ of
$E_{n,\lambda}$.
\item
Finally compute the row canonical form of $E_{n,\lambda}$.
\end{enumerate}

\begin{remark}
The skeletons in $\TriAss(n)$ are in bijection with the equivalence classes of skeletons in $\BBB(n)$:
equivalence is determined by the (consequences of the) triassociative relations.
Computing this equivalence relation is only practical in low arities.
In general, we use the explicit description of the free trioid on one generator \cite{LRtrialgebras}.
\end{remark}

\begin{definition}
The \textbf{toprank} of $\mathrm{RCF}(E_{n,\lambda})$ is the number $\mathrm{top}(n,\lambda)$ of rows with leading 1s
to the left of the vertical line in \eqref{Enlambda}.
That is, $\mathrm{top}(n,\lambda)$ is the largest $i$ such that the leading 1 of row $i$ is in column
$j \le t(n) d_\lambda$.
Every row $i' > i$ has only 0s to the left of the vertical line.
The \textbf{allrank} of $\mathrm{RCF}(E_{n,\lambda})$ is the number
$\mathrm{all}(n,\lambda) = b(n) d_\lambda - \mathrm{top}(n,\lambda)$
of rows with leading 1s to the right of the vertical line.
\end{definition}

\begin{lemma}
The last $\mathrm{all}(n,\lambda)$ rows of $\mathrm{RCF}(E_{n,\lambda})$ are independent $S_n$-module generators
for the isotypic component of partition $\lambda$ in the kernel of $E(n)$:
the sum of all irreducible submodules of $\mathrm{kernel}(E(n))$ isomorphic to $[\lambda]$.
\end{lemma}

\begin{definition}
$\mathrm{ALL}(n,\lambda)$, the \textbf{matrix of all identities} for arity $n$, partition $\lambda$,
is the lower right block of $\mathrm{RCF}(E_{n,\lambda})$, with upper left corner in row $\mathrm{top}(n,\lambda) + 1$
and column $t(n) d_\lambda + 1$.
\end{definition}


\subsubsection{Old relations; symmetries of the skeletons}

For each partition $\lambda$ of $n$, the nullspace of $E_{n,\lambda}$ contains \emph{all} polynomial identities
in the isotypic component for $\lambda$ which are satisfied by the Jordan product $\circ$ and diproduct $\bullet$
in every triassociative algebra.
In general, many of these are consequences of known identities of lower arity.

We must also consider the symmetries of the skeletons in $\BWS(n)$.
If $n$ is small, we can compute a monomial basis and avoid representation theory:
to each skeleton, we apply all permutations of the $n$ variables, and retain a monomial only if
it equals its normal form.
For example, let $n = 3$ and consider $( x \,\circ\, x ) \bullet x$:
commutativity of $\circ$ implies that the six permutations of $a,b,c$ produce only three normal forms:
$( a \,\circ\, b ) \bullet c$, $( a \,\circ\, c ) \bullet b$, $( b \,\circ\, c ) \bullet a$.
If $n$ is large, so that it is not practical to compute a monomial basis, then we use representation theory,
and so we must encode the symmetries in some other way.
For a small example, from the skeleton $( x \circ x ) \bullet x$ we construct a
multilinear identity which can be processed using representation theory:
$( x_1 \circ x_2 ) \bullet x_3 - ( x_2 \circ x_1 ) \bullet x_3 \equiv 0$.

\begin{definition}
We call the identities of this form the \textbf{symmetries} of the skeletons.
In arity $n$, each symmetry has the form $m_1 - m_2 \equiv 0$, where $m_1, m_2$ are monomials of arity $n$
with the same skeleton; $m_1$ has the identity permutation of the variables, and $m_2$ has a permutation of order 2;
and the commutativity of $\circ$ implies that $m_1 = m_2$.
\end{definition}

\begin{lemma}
The relation between multilinear monomials, the skeletons, and their symmetries, is given by this equation,
where $s(\beta)$ is the number of symmetries of skeleton $\beta$:
\[
\dim \BW(n) = \sum_{\beta \,\in\, \BWS(n)} \frac{n!}{2^{s(\beta)}}
\]
\end{lemma}

\begin{definition}
\label{symcondef}
Recall that $b(n)$ is the size of $\BWS(n)$: the total number of $\BW$ skeletons in arity $n$.
Let $\mathrm{sym}(n)$ be the total number of symmetries over all skeletons in arity $n$.
Let $\mathrm{con}(n)$ be the total number of consequences in arity $n$ of the known identities of lower arities.
Let $K_{n,\lambda}$ be the matrix of size $( \mathrm{sym}(n) {+} \mathrm{con}(n) ) d_\lambda \times b(n) d_\lambda$,
consisting of $d_\lambda \times d_\lambda$ blocks; the block in position $(i,j)$ is $R_\lambda( Y_{i,j} )$
where $Y_{i,j}$ is the component of symmetry $i$ in skeleton $j$ for $i = 1, \dots, \mathrm{sym}(n)$, and
$Y_{\mathrm{sym}(n)+i,j}$ is the component of consequence $i$ in skeleton $j$ for $i = 1, \dots, \mathrm{con}(n)$.
The \textbf{matrix of old identities} for arity $n$ and paritition $\lambda$ is defined by
$\mathrm{OLD}(n,\lambda) = \mathrm{RCF}(K_{n,\lambda})$.
\end{definition}


\subsection{Relations of arity $\ge 5$: a result, a problem, and two conjectures}

\begin{proposition}
\label{JTA-arity56}
Every multilinear polynomial identity of arity $\le 6$ relating the Jordan product and diproduct
in every triassociative algebra is a consequence of the defining identities for Jordan trialgebras
(Definition \ref{triJordandef}).
\end{proposition}

\begin{proof}
For $n = 5, 6$ and all partitions $\lambda$ of $n$, we computed $\mathrm{ALL}(n,\lambda)$
and $\mathrm{OLD}(n,\lambda)$, which are both in row canonical form, and found that they were equal in every case.
\end{proof}

\begin{problem}
Determine whether or not special identities of arity 8 exist for the Jordan diproduct in the free
diassociative algebra; such identities (if they exist) are called special identities for Jordan dialgebras.
The existence of such identities for the pre-Jordan product in the free dendriform algebra has been
established \cite{BM}.
\end{problem}

\begin{conjecture}
Over a field $\mathbb{F}$ of characteristic 0 or $p > 7$, every multilinear polynomial identity of
arity $\le 7$ satisfied by the Jordan product and diproduct in every triassociative algebra is a
consequence of the defining identities for Jordan trialgebras.
\end{conjecture}

\begin{conjecture}
There exist identities of arity 8 the Jordan product and diproduct in every triassociative algebra
which are not consequences of the identities defining Jordan trialgebras, the (linearizations of) the
Glennie identities for Jordan algebras \cite{Glennie}, and the multilinear special identities of
arity 8 for Jordan dialgebras (if they exist).
\end{conjecture}


\section{Enumeration of association types and skeletons}

We consider two integer sequences:
the number of skeletons of arity $n$ in the symmetric operad $\BW$ generated
by two binary operations, one commutative, one noncommutative;
the number of multilinear monomials of arity $n$ in $\BW$, which is $\dim\BW(n)$.


\subsection{One commutative operation}

The number of association types is sequence A001190 in the OEIS (\url{oeis.org}),
the Wedderburn-Etherington numbers \cite{Etherington,Wedderburn}:
\[
1, \, 1, \, 1, \, 2, \, 3, \, 6, \, 11, \, 23, \, 46, \, 98, \, 207, \, 451, \, 983, \, 2179, \, 4850, \,
10905, \, 24631, \, 56011, \, \dots
\]
This sequence also enumerates complete rooted binary trees with $n$ leaves, up to abstract graph isomorphism,
so that (for example) the two trees with three (unlabelled) leaves are isomorphic.
The corresponding number of multilinear monomials is A001147, the double factorial of odd numbers:
$(2n{-}1)!! = 1 \cdot 3 \cdot 5 \cdots (2n{-}1)$.
This is also the number of complete rooted binary trees with leaf labels $1,\dots,n$ up to abstract isomorphism.


\subsection{One noncommutative operation}

The number of association types is the (shifted) Catalan number $C(n{-}1)$ where $C(n) = (2n)!/(n!(n{+}1)!)$.
(The shift is required since $n$ in $C(n)$ is the number of operations but for us $n$ is the number of arguments.)
This sequence A000108 also enumerates complete rooted binary \emph{plane} trees with $n$ leaves,
so that the two trees with three (unlabelled) leaves are \emph{not} isomorphic.
In this case, the operation has no symmetry, so the number of multilinear monomials is simply
$n!C(n{-}1) = (2n{-}2)!/(n{-}1)!$,
which is the quadruple factorial $(4n{-}2)!!!! = 2 \cdot 6 \cdot 10 \cdots (4n-2)$.


\subsection{Two operations: one commutative, one noncommutative}

Computational enumeration of the skeletons produced the following sequence, A276277 in the OEIS:
\[
1, \, 2, \, 6, \, 25, \, 111, \, 540, \, 2736, \, 14396, \, 77649, \, 427608, \, 2392866, \, 13570386, \, 77815161, \, \dots
\]
The number of multilinear monomials is the sextuple factorial, sequence A011781:
\[
(6n{-}3)!!!!!! = \prod_{k=1}^{n-1} (6k-3) \implies
1, \, 3, \, 27, \, 405, \, 8505, \,  229635, \, 7577955, \, 295540245, \, \dots
\]

\begin{problem}
For $p, q \ge 1$ let $\BW_{p,q}$ be the free symmetric operad generated by $p$ commutative
and $q$ noncommutative binary operations.
For $n \ge 1$ determine the number of skeletons and the number of multilinear monomials in $\BW_{p,q}(n)$.
\end{problem}


\section{Post-Jordan algebras}
\label{postjordanalgebrasection}

In this section we determine the multilinear polynomial identities of arity $\le 6$ satisfied by
the Jordan and pre-Jordan products
$a \,\circ\, b = a \curlywedge b + b \curlywedge a$ and $a \bullet b = a \prec b  + b \succ a$
in the free tridendriform algebra.
In addition to the commutativity of $\circ$, there are no new identities in arity 3, a nonzero $S_4$-module of new identities in arity 4 for which we find a minimal set of generators, and no new identities in arities 5 and 6.
The commutativity of $\circ$ together with the new identities in arity 4 define post-Jordan algebras.


\subsection{Gr\"obner-Shirshov basis for the tridendriform operad}

The computations for post-Jordan algebras are very similar to those for Jordan trialgebras; in particular,
the domain $\BW(n)$ of the expansion map $E(n)$ is the same.
However, the codomain of the expansion map is no longer $\STriAss$ but its Koszul dual
$\STriDend \cong \STriAss^!$.
The defining relations for $\TriAss$ are monomial relations (Definition \ref{triassdef}):
this allows us to compute normal forms in $\TriAss$ very simply in terms of equivalence relations
on $\BBB$ monomials.
On the other hand, not all the defining relations for $\TriDend$ are monomial relations
(Definition \ref{triassdef}).
Hence computing normal forms in $\TriDend$ requires first determining a Gr\"obner-Shirshov (GS) basis
from the defining relations, and this requires fixing a total order on the operations.
Checking all six possibilities, we find that the smallest GS basis for $\TriDend$
comes from the order $\curlywedge, \prec, \succ$ where the commutative operation comes first.
This GS basis for $\TriDend$ is almost the same as the original defining relations,
except that the relations appear in a different order, as do the terms within each relation.

\begin{lemma}
\label{TDGSlemma}
Starting with the relations in Definition \ref{triassdef},
and assuming the order of operations $\curlywedge, \prec, \succ$,
we obtain the following Gr\"obner-Shirshov basis for $\TriDend$:
\[
\begin{array}{lll}
\multicolumn{2}{l}{
  ( a \prec b ) \prec c
- a \prec ( b \succ c )
- a \prec ( b \prec c )
- a \prec ( b \curlywedge c ),
}
&\quad
  ( a \succ b ) \prec c
- a \succ ( b \prec c ),
\\
\multicolumn{2}{l}{
  ( a \succ b ) \succ c
+ ( a \prec b ) \succ c
+ ( a \curlywedge b ) \succ c
- a \succ ( b \succ c ),
}
&\quad
  ( a \succ b ) \curlywedge c
- a \succ ( b \curlywedge c ),
\\
  ( a \prec b ) \curlywedge c
- a \curlywedge ( b \succ c ),
&\quad
  ( a \curlywedge b ) \prec c
- a \curlywedge ( b \prec c ),
&\quad
  ( a \curlywedge b ) \curlywedge c
- a \curlywedge ( b \curlywedge c ) .
\end{array}
\]
\end{lemma}

\begin{proof}
Gr\"obner bases for operads were introduced in \cite{DK}; the special case of nonsymmetric operads
appears in \cite{DV}.
Similar computations are explained in detail in \cite{Madariaga}.
\end{proof}

We use this GS basis for $\TriDend$ to compute normal forms of nonsymmetric tridendriform polynomials.
This also applies in the symmetric case, since $\STriDend$ is the symmetrization of $\TriDend$.
A monomial $m \in \STriDend(n)$ has a skeleton $s \in \TriDend(n)$ and a permutation $p \in S_n$
of the arguments $x_1, \dots, x_n$.
To find the normal form of $m$, we compute the normal form of $s$,
and then replace the arguments $x_{p(1)}, \dots, x_{p(n)}$.
For details on nonsymmetric operads and their Gr\"obner bases, see \cite{BD}.


\subsection{Relations of arity $\le 4$}

\begin{lemma}
Every multilinear polynomial identity of arity 3 for the Jordan and pre-Jordan products $\circ, \bullet$
in the free tridendriform algebra is a consequence of the commutativity of $\circ$.
\end{lemma}

\begin{proof}
The expansion matrix $E_3$ has size $66 \times 27$: see Table \ref{PJAexpansion3}, where the matrix is split
into top and bottom halves.
Following the proof of Lemma \ref{jordantrialgebralemma},
the columns correspond to the ordered monomial basis of the domain $\BW(3)$ of the expansion map $E(3)$;
see the proof of Lemma \ref{lietriops} where the symbols $\curlywedge, \prec$ are used instead of $\circ, \bullet$.
As for the codomain, since $\dim \BBB(3) = 18$, and the GS basis of Lemma \ref{TDGSlemma} has 7 elements,
there are 7 leading monomials which are linear combinations of the other 11 monomials,
so $\dim \TriDend(3) = 11$.
To each nonsymmetric basis monomial we apply all 6 permutations of the arguments;
this gives the ordered monomial basis of $\STriDend(3)$ corresponding to the 66 rows of $E_3$.
An easy calculation shows that $E_3$ has full rank and so its nullspace is 0.
\end{proof}

\begin{table}[ht]
\begin{center}
\scriptsize
$
\left[
\begin{array}{@{}c@{}c@{}c@{}c@{}c@{}c@{}c@{}c@{}c@{}c@{}c@{}c@{}c@{}c@{}c@{}c@{}c@{}c@{}c@{}c@{}c@{}c@{}c@{}c@{}c@{}c@{}c@{}}
+ & . & + & . & . & . & . & . & . & . & . & . & . & . & . & . & . & . & . & . & . & . & . & . & . & . & . \\[-1mm]
. & + & + & . & . & . & . & . & . & . & . & . & . & . & . & . & . & . & . & . & . & . & . & . & . & . & . \\[-1mm]
+ & + & . & . & . & . & . & . & . & . & . & . & . & . & . & . & . & . & . & . & . & . & . & . & . & . & . \\[-1mm]
. & + & + & . & . & . & . & . & . & . & . & . & . & . & . & . & . & . & . & . & . & . & . & . & . & . & . \\[-1mm]
+ & + & . & . & . & . & . & . & . & . & . & . & . & . & . & . & . & . & . & . & . & . & . & . & . & . & . \\[-1mm]
+ & . & + & . & . & . & . & . & . & . & . & . & . & . & . & . & . & . & . & . & . & . & . & . & . & . & . \\[-1mm]
. & . & . & . & . & . & + & . & . & + & . & . & . & . & . & . & . & . & . & . & . & . & . & . & . & . & . \\[-1mm]
. & . & . & . & . & . & . & . & + & . & + & . & . & . & . & . & . & . & . & . & . & . & . & . & . & . & . \\[-1mm]
. & . & . & . & + & . & . & . & . & + & . & . & . & . & . & . & . & . & . & . & . & . & . & . & . & . & . \\[-1mm]
. & . & . & . & . & . & . & + & . & . & . & + & . & . & . & . & . & . & . & . & . & . & . & . & . & . & . \\[-1mm]
. & . & . & + & . & . & . & . & . & . & + & . & . & . & . & . & . & . & . & . & . & . & . & . & . & . & . \\[-1mm]
. & . & . & . & . & + & . & . & . & . & . & + & . & . & . & . & . & . & . & . & . & . & . & . & . & . & . \\[-1mm]
. & . & . & + & . & . & . & . & + & . & . & . & . & . & . & . & . & . & . & . & . & . & . & . & . & . & . \\[-1mm]
. & . & . & . & + & . & + & . & . & . & . & . & . & . & . & . & . & . & . & . & . & . & . & . & . & . & . \\[-1mm]
. & . & . & . & . & + & . & + & . & . & . & . & . & . & . & . & . & . & . & . & . & . & . & . & . & . & . \\[-1mm]
. & . & . & . & + & . & + & . & . & . & . & . & . & . & . & . & . & . & . & . & . & . & . & . & . & . & . \\[-1mm]
. & . & . & . & . & + & . & + & . & . & . & . & . & . & . & . & . & . & . & . & . & . & . & . & . & . & . \\[-1mm]
. & . & . & + & . & . & . & . & + & . & . & . & . & . & . & . & . & . & . & . & . & . & . & . & . & . & . \\[-1mm]
. & . & . & . & . & . & . & . & . & . & . & . & + & . & . & . & . & . & + & . & . & . & . & . & . & . & . \\[-1mm]
. & . & . & . & . & . & . & . & . & . & . & . & . & + & . & . & . & . & + & . & . & . & . & . & . & . & . \\[-1mm]
. & . & . & . & . & . & . & . & . & . & . & . & . & . & + & . & . & . & . & + & . & . & . & . & . & . & . \\[-1mm]
. & . & . & . & . & . & . & . & . & . & . & . & . & . & . & + & . & . & . & + & . & . & . & . & . & . & . \\[-1mm]
. & . & . & . & . & . & . & . & . & . & . & . & . & . & . & . & + & . & . & . & + & . & . & . & . & . & . \\[-1mm]
. & . & . & . & . & . & . & . & . & . & . & . & . & . & . & . & . & + & . & . & + & . & . & . & . & . & . \\[-1mm]
. & . & . & . & . & . & . & . & . & . & . & . & + & . & . & . & . & . & . & . & . & + & . & . & . & . & . \\[-1mm]
. & . & . & . & . & . & . & . & . & . & . & . & . & + & . & . & . & . & . & . & . & . & + & . & . & . & . \\[-1mm]
. & . & . & . & . & . & . & . & . & . & . & . & . & . & + & . & . & . & . & . & . & . & . & + & . & . & . \\[-1mm]
. & . & . & . & . & . & . & . & . & . & . & . & . & . & . & + & . & . & . & . & . & . & . & . & + & . & . \\[-1mm]
. & . & . & . & . & . & . & . & . & . & . & . & . & . & . & . & + & . & . & . & . & . & . & . & . & + & . \\[-1mm]
. & . & . & . & . & . & . & . & . & . & . & . & . & . & . & . & . & + & . & . & . & . & . & . & . & . & + \\[-1mm]
. & . & . & . & . & . & . & . & . & . & . & . & + & . & . & . & . & . & . & . & . & . & + & . & . & . & . \\[-1mm]
. & . & . & . & . & . & . & . & . & . & . & . & . & + & . & . & . & . & . & . & . & + & . & . & . & . & . \\[-1mm]
. & . & . & . & . & . & . & . & . & . & . & . & . & . & + & . & . & . & . & . & . & . & . & . & + & . & .
\end{array}
\right]
\qquad\qquad\qquad
\left[
\begin{array}{@{}c@{\;}c@{\;}c@{}c@{}c@{}c@{}c@{}c@{}c@{}c@{}c@{}c@{}c@{}c@{}c@{}c@{}c@{}c@{}c@{}c@{}c@{}c@{}c@{}c@{}c@{}c@{}c@{}}
. & . & . & . & . & . & . & . & . & . & . & . & . & . & . & + & . & . & . & . & . & . & . & + & . & . & . \\[-1mm]
. & . & . & . & . & . & . & . & . & . & . & . & . & . & . & . & + & . & . & . & . & . & . & . & . & . & + \\[-1mm]
. & . & . & . & . & . & . & . & . & . & . & . & . & . & . & . & . & + & . & . & . & . & . & . & . & + & . \\[-1mm]
. & . & . & . & . & + & . & . & . & . & . & + & . & . & . & . & . & . & . & . & . & . & . & . & . & . & . \\[-1mm]
. & . & . & . & . & . & . & + & . & . & . & + & . & . & . & . & . & . & . & . & . & . & . & . & . & . & . \\[-1mm]
. & . & . & + & . & . & . & . & . & . & + & . & . & . & . & . & . & . & . & . & . & . & . & . & . & . & . \\[-1mm]
. & . & . & . & . & . & . & . & + & . & + & . & . & . & . & . & . & . & . & . & . & . & . & . & . & . & . \\[-1mm]
. & . & . & . & + & . & . & . & . & + & . & . & . & . & . & . & . & . & . & . & . & . & . & . & . & . & . \\[-1mm]
. & . & . & . & . & . & + & . & . & + & . & . & . & . & . & . & . & . & . & . & . & . & . & . & . & . & . \\[-1mm]
. & . & . & . & . & . & . & . & . & . & . & . & . & . & + & + & . & . & . & . & . & . & . & . & . & . & . \\[-1mm]
. & . & . & . & . & . & . & . & . & . & . & . & . & . & . & . & + & + & . & . & . & . & . & . & . & . & . \\[-1mm]
. & . & . & . & . & . & . & . & . & . & . & . & + & + & . & . & . & . & . & . & . & . & . & . & . & . & . \\[-1mm]
. & . & . & . & . & . & . & . & . & . & . & . & . & . & . & . & + & + & . & . & . & . & . & . & . & . & . \\[-1mm]
. & . & . & . & . & . & . & . & . & . & . & . & + & + & . & . & . & . & . & . & . & . & . & . & . & . & . \\[-1mm]
. & . & . & . & . & . & . & . & . & . & . & . & . & . & + & + & . & . & . & . & . & . & . & . & . & . & . \\[-1mm]
. & . & . & . & . & . & . & . & . & . & . & . & . & . & . & . & . & + & . & . & . & . & . & . & . & . & + \\[-1mm]
. & . & . & . & . & . & . & . & . & . & . & . & . & . & . & + & . & . & . & . & . & . & . & . & + & . & . \\[-1mm]
. & . & . & . & . & . & . & . & . & . & . & . & . & . & . & . & + & . & . & . & . & . & . & . & . & + & . \\[-1mm]
. & . & . & . & . & . & . & . & . & . & . & . & . & + & . & . & . & . & . & . & . & . & + & . & . & . & . \\[-1mm]
. & . & . & . & . & . & . & . & . & . & . & . & . & . & + & . & . & . & . & . & . & . & . & + & . & . & . \\[-1mm]
. & . & . & . & . & . & . & . & . & . & . & . & + & . & . & . & . & . & . & . & . & + & . & . & . & . & . \\[-1mm]
. & . & . & . & . & . & . & . & . & . & . & . & . & . & . & . & . & . & . & . & + & . & . & . & . & . & - \\[-1mm]
. & . & . & . & . & . & . & . & . & . & . & . & . & . & . & . & . & . & . & + & . & . & . & . & - & . & . \\[-1mm]
. & . & . & . & . & . & . & . & . & . & . & . & . & . & . & . & . & . & . & . & + & . & . & . & . & - & . \\[-1mm]
. & . & . & . & . & . & . & . & . & . & . & . & . & . & . & . & . & . & + & . & . & . & - & . & . & . & . \\[-1mm]
. & . & . & . & . & . & . & . & . & . & . & . & . & . & . & . & . & . & . & + & . & . & . & - & . & . & . \\[-1mm]
. & . & . & . & . & . & . & . & . & . & . & . & . & . & . & . & . & . & + & . & . & - & . & . & . & . & . \\[-1mm]
. & . & . & . & . & . & . & . & . & . & . & . & . & . & . & . & . & . & . & . & . & . & . & . & . & + & - \\[-1mm]
. & . & . & . & . & . & . & . & . & . & . & . & . & . & . & . & . & . & . & . & . & . & . & + & - & . & . \\[-1mm]
. & . & . & . & . & . & . & . & . & . & . & . & . & . & . & . & . & . & . & . & . & . & . & . & . & - & + \\[-1mm]
. & . & . & . & . & . & . & . & . & . & . & . & . & . & . & . & . & . & . & . & . & + & - & . & . & . & . \\[-1mm]
. & . & . & . & . & . & . & . & . & . & . & . & . & . & . & . & . & . & . & . & . & . & . & - & + & . & . \\[-1mm]
. & . & . & . & . & . & . & . & . & . & . & . & . & . & . & . & . & . & . & . & . & - & + & . & . & . & .
\end{array}
\right]
$
\end{center}
\smallskip
\caption{Top and bottom of post-Jordan expansion matrix in arity 3}
\label{PJAexpansion3}
\end{table}

\normalsize

\begin{theorem}
\label{postjordanalgebratheorem}
Over a field $\mathbb{F}$ of characteristic 0 or $p > 4$, every multilinear polynomial identity
of arity $\le 4$ satisfied by the Jordan and pre-Jordan products in the free tridendriform algebra
is a consequence of the commutativity of $\circ$ and the (linearizations of the) following 7 identities
of arity 4 (there are no new identities of arity 3):
\begin{align*}
&
  ( ( c \bullet a ) \,\circ\,   b ) \,\circ\,   d  
+ ( ( d \bullet a ) \,\circ\,   b ) \,\circ\,   c  
- ( ( b \,\circ\,   c ) \bullet a ) \,\circ\,   d  
- ( ( b \,\circ\,   d ) \bullet a ) \,\circ\,   c  
\\[-1mm]
&
\qquad
- ( ( c \,\circ\,   d ) \bullet a ) \,\circ\,   b  
+ ( ( c \,\circ\,   d ) \,\circ\,   b ) \bullet a  
\equiv 0,
\\[-1mm]
&
  ( ( b \,\circ\,   c ) \bullet a ) \,\circ\,   d  
+ ( ( b \,\circ\,   d ) \bullet a ) \,\circ\,   c  
+ ( ( c \,\circ\,   d ) \bullet a ) \,\circ\,   b  
- ( b \,\circ\,   c ) \,\circ\,   ( d \bullet a )  
\\[-1mm]
&
\qquad
- ( b \,\circ\,   d ) \,\circ\,   ( c \bullet a )  
- ( c \,\circ\,   d ) \,\circ\,   ( b \bullet a )  
\equiv 0,
\\[-1mm]
&
  ( ( a \,\circ\,   b ) \,\circ\,   d ) \,\circ\,   c  
+ ( ( a \,\circ\,   c ) \,\circ\,   d ) \,\circ\,   b  
+ ( ( b \,\circ\,   c ) \,\circ\,   d ) \,\circ\,   a  
- ( a \,\circ\,   b ) \,\circ\,   ( c \,\circ\,   d )  
\\[-1mm]
&
\qquad
- ( a \,\circ\,   c ) \,\circ\,   ( b \,\circ\,   d )  
- ( a \,\circ\,   d ) \,\circ\,   ( b \,\circ\,   c )  
\equiv 0,
\\[-1mm]
&
  ( ( c \bullet a ) \bullet b ) \,\circ\,   d  
+ ( ( d \bullet a ) \bullet b ) \,\circ\,   c  
- ( d \bullet ( a \,\circ\,   b ) ) \,\circ\,   c  
- ( d \bullet ( a \bullet b ) ) \,\circ\,   c  
\\[-1mm]
&
\qquad
- ( d \bullet ( b \bullet a ) ) \,\circ\,   c  
- ( ( c \bullet a ) \,\circ\,   d ) \bullet b  
- ( ( c \bullet b ) \,\circ\,   d ) \bullet a  
+ ( ( c \,\circ\,   d ) \bullet b ) \bullet a  
\equiv 0,
\\[-1mm]
&
  ( d \bullet ( a \,\circ\,   b ) ) \,\circ\,   c  
+ ( d \bullet ( a \bullet b ) ) \,\circ\,   c  
+ ( d \bullet ( b \bullet a ) ) \,\circ\,   c  
- ( c \bullet a ) \,\circ\,   ( d \bullet b )  
\\[-1mm]
&
\qquad
- ( c \bullet b ) \,\circ\,   ( d \bullet a )  
+ ( ( c \bullet a ) \,\circ\,   d ) \bullet b  
+ ( ( c \bullet b ) \,\circ\,   d ) \bullet a  
- ( c \,\circ\,   d ) \bullet ( a \,\circ\,   b )  
\\[-1mm]
&
\qquad
- ( c \,\circ\,   d ) \bullet ( a \bullet b )  
- ( c \,\circ\,   d ) \bullet ( b \bullet a )  
\equiv 0,
\\[-1mm]
&
  ( d \bullet ( a \,\circ\,   b ) ) \bullet c  
+ ( d \bullet ( a \,\circ\,   c ) ) \bullet b  
+ ( d \bullet ( b \,\circ\,   c ) ) \bullet a  
+ ( d \bullet ( a \bullet b ) ) \bullet c  
\\[-1mm]
&
\qquad
+ ( d \bullet ( a \bullet c ) ) \bullet b  
+ ( d \bullet ( b \bullet a ) ) \bullet c  
+ ( d \bullet ( b \bullet c ) ) \bullet a  
+ ( d \bullet ( c \bullet a ) ) \bullet b  
\\[-1mm]
&
\qquad
+ ( d \bullet ( c \bullet b ) ) \bullet a  
- ( d \bullet a ) \bullet ( b \,\circ\,   c )  
- ( d \bullet b ) \bullet ( a \,\circ\,   c )  
- ( d \bullet c ) \bullet ( a \,\circ\,   b )  
\\[-1mm]
&
\qquad
- ( d \bullet a ) \bullet ( b \bullet c )  
- ( d \bullet a ) \bullet ( c \bullet b )  
- ( d \bullet b ) \bullet ( a \bullet c )  
- ( d \bullet b ) \bullet ( c \bullet a )  
\\[-1mm]
&
\qquad
- ( d \bullet c ) \bullet ( a \bullet b )  
- ( d \bullet c ) \bullet ( b \bullet a )  
\equiv 0,
\\[-1mm]
&
  ( ( d \bullet a ) \bullet c ) \bullet b  
+ ( ( d \bullet b ) \bullet c ) \bullet a  
- ( d \bullet ( a \,\circ\,   b ) ) \bullet c  
- ( d \bullet ( a \,\circ\,   c ) ) \bullet b  
\\[-1mm]
&
\qquad
- ( d \bullet ( b \,\circ\,   c ) ) \bullet a  
- ( d \bullet ( a \bullet b ) ) \bullet c  
- ( d \bullet ( a \bullet c ) ) \bullet b  
- ( d \bullet ( b \bullet a ) ) \bullet c  
\\[-1mm]
&
\qquad
- ( d \bullet ( b \bullet c ) ) \bullet a  
- ( d \bullet ( c \bullet a ) ) \bullet b  
- ( d \bullet ( c \bullet b ) ) \bullet a  
+ d \bullet ( ( a \,\circ\,   b ) \,\circ\,   c )  
\\[-1mm]
&
\qquad
+ d \bullet ( ( a \bullet b ) \,\circ\,   c )  
+ d \bullet ( ( b \bullet a ) \,\circ\,   c )  
+ d \bullet ( ( a \,\circ\,   b ) \bullet c )  
+ d \bullet ( ( a \bullet b ) \bullet c )  
\\[-1mm]
&
\qquad
+ d \bullet ( ( b \bullet a ) \bullet c )  
+ d \bullet ( c \bullet ( a \,\circ\,   b ) )  
+ d \bullet ( c \bullet ( a \bullet b ) )  
+ d \bullet ( c \bullet ( b \bullet a ) )  
\equiv 0.
\end{align*}
\end{theorem}

\begin{proof}
All the techniques have already been discussed, so we will be very brief.
The expansion matrix $E_4$ has size $1080 \times 405$ and rank 345;
its nullity is 60, and every nonzero vector in the nullspace is a new identity,
since there are no consequences from arity 3.
We extract the canonical basis vectors for the nullspace and sort the identities
by increasing number of terms.
We find a subset of seven identities which generates the nullspace as an $S_4$-module;
none belongs to the $S_4$-module generated by the others.
These identities have 6, 6, 6, 8, 10, 18, 20 terms.
The first is the linearized Jordan identity which contains only the operation $\circ$.
Every other identity contains both operations.
If we remove every term containing $\circ$ from these seven identities, then the first five
identities become 0, and the last two identities become the defining identities in arity 4 for
pre-Jordan algebras.
\end{proof}

\begin{remark}
In a Jordan trialgebra, the Jordan product $\circ$ is commutative and satisfies the Jordan identity;
the Jordan diproduct $\bullet$ satisfies the defining relations for Jordan dialgebras.
Thus a Jordan trialgebra is a sum or split extension (roughly speaking) of a Jordan algebra by a Jordan dialgebra.
In a post-Jordan algebra, the Jordan product $\circ$ is commutative and satisfies the Jordan identity;
however, the pre-Jordan product $\bullet$ \emph{does not} satisfy the defining identities for
pre-Jordan algebras.
In order to obtain the pre-Jordan identities, we must set the Jordan product to zero: remove every term
containing $\circ$.
Thus a post-Jordan algebra is a \emph{non-split} extension of a Jordan algebra by a pre-Jordan algebra.
\end{remark}

\begin{definition}
\label{postJordandef}
Over a field $\mathbb{F}$ of characteristic 0 or $p > 4$, a vector space $J$ with bilinear operations
$\circ$ and $\bullet$ is a \textbf{post-Jordan algebra} if $\circ$ is commutative and $\circ, \bullet$
together satisfy the multilinear polynomial identities of Theorem \ref{postjordanalgebratheorem}.
In particular, $( J, \,\circ\, )$ is a Jordan algebra.
The corresponding symmetric operad is denoted $\PostJor$.
\end{definition}


\subsection{Trisuccessors}
\label{trisuccessorsection}

The operad $\PostJor$ governing post-Jordan algebras may be obtained using the techniques of
\cite{PBGN}.
As in \S\ref{triplicatorsection}, we start from the linearized Jordan identity represented
in terms of tree monomials.
We apply Algorithm \ref{tripalgorithm} but with the difference that $\ast$ no longer
represents the union $\{ \circ_1, \circ_2, \circ_3 \}$ but the sum $\circ_1 + \circ_2 + \circ_3$.
As before, we consider the example $L = \{ a,c \}$.
In the resulting tree polynomial, term 5 has a vertex labelled $\ast$, which is now replaced by
the sum of three terms $5'$, $5''$, $5'''$ obtained by substituting (operations) 1, 2, 3 for $\ast$;
the result is a tree polynomial with eight terms.
As in \S\ref{triplicatorsection}, the three new operations satisfy symmetries:
$a \,\circ_1\, b \equiv b \,\circ_3\, a$ and $a \,\circ_2\, b \equiv b \,\circ_2\, a$.
These allow us to rewrite the tree polynomial using only two operations:

\vspace{-3mm}
\begin{figure}[h]
\centering
\includegraphics[scale=1]{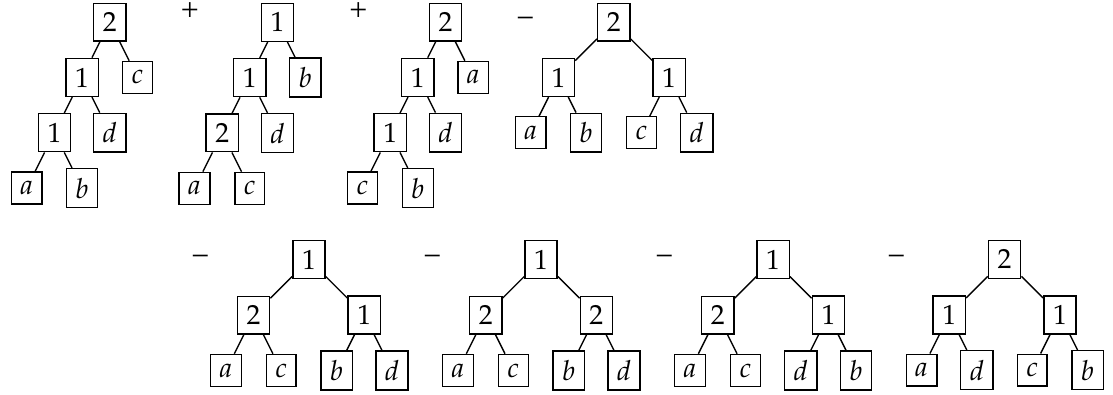}
\vspace{-2mm}
\end{figure}

\normalsize

\noindent
We replace operation 2 by $\circ$, and 1 by $\bullet$, and convert the tree monomials
to parentheses and permutations.
We perform this calculation for all nonempty $L \subseteq \{a,b,c,d\}$
and obtain multilinear identities which can be compared directly to those of
Theorem \ref{postjordanalgebratheorem}.

\begin{definition}
The operad defined by the binary operations $\circ$ and $\bullet$, where $\circ$ is commutative and
$\bullet$ has no symmetry, satisfying the multilinear relations obtained from all possible applications
of the trisuccessor algorithm to the linearized Jordan identity,
is denoted $\myfont{TriSucJor}$ and is called the \textbf{trisuccessor} of the Jordan operad $\Jor$.
\end{definition}

\begin{proposition}
Let $\PostJor$ be the symmetric operad governing post-Jordan algebras,
generated by the binary operations $\circ$ (commutative) and $\bullet$ (no symmetry)
satisfying the multilinear identities of Theorem \ref{postjordanalgebratheorem}.
Let $\myfont{TriSucJor}$ be the symmetric operad generated by the same operations but satisfying
the multilinear identities obtained by applying the trisuccessor algorithm to the linearized Jordan identity.
These two sets of identities generate the same subquotient $S_4$-module of $\BW(4)$, and hence
the two operads are isomorphic.
\end{proposition}


\subsection{Relations of arity $\ge 5$: a result and two conjectures}

\begin{proposition}
\label{PJA-arity56}
Every multilinear polynomial identity of arity $\le 6$ relating the Jordan and pre-Jordan products
in every tridendriform algebra is a consequence of the defining identities for post-Jordan algebras.
\end{proposition}

\begin{proof}
Similar to the proof of Proposition \ref{JTA-arity56}.
\end{proof}

\begin{conjecture}
Over a field $\mathbb{F}$ of characteristic 0 or $p > 7$, every multilinear polynomial identity of
arity $\le 7$ satisfied by the Jordan and pre-Jordan products in every tridendriform algebra is a
consequence of the defining identities for post-Jordan algebras.
\end{conjecture}

\begin{conjecture}
There exist identities of arity 8 which are not consequences of the defining identities for post-Jordan
algebras, the Glennie identities for Jordan algebras \cite{Glennie}, and the special identities of arity 8
for pre-Jordan algebras \cite{BM}.
\end{conjecture}


\section{Concluding remarks}
\label{sectionconclusion}


\subsection{Koszul duality for nonquadratic operads}
\label{nonquadratickoszulduality}

Koszul duality for associative algebras has been extended \cite{DV} from the quadratic case ($n = 2$)
to the $n$-homogeneous case $(n > 2)$.
If a similar extension exists of Koszul duality for operads, it could be applied to the operads
governing the Jordan structures studied in this paper.
There is another approach: introduce a new ternary operation to lower the weight of the relations.
Consider an operad generated by a commutative binary operation $ab$.
If we define a new ternary operation by $(a,b,c) = (ab)c$ then every binary monomial of weight 2
equals a ternary monomial of weight 1, and every binary monomial of weight 3 equals a binary-ternary
monomial of weight 2.
Any cubic relation in the binary operation can be rewritten as a quadratic relation in both operations.
For example, consider the linearized Jordan identity stated at the start of \S\ref{triplicatorsection}.
We replace monomials of the forms $((wx)y)z$ and $(wx)(yz)$ by $(wx,y,z)$ and $(w,x,yz)$ respectively:
\[
(ab,d,c) + (ac,d,b) + (bc,d,a) - (a,b,cd) - (a,c,bd) - (b,c,ad) \equiv 0.
\]
This relation is quadratic: each term involves one binary operation and one ternary operation.
However, we must also include the relation $(a,b,c) - (ab)c \equiv 0$, which is homogeneous in the arity,
but inhomogeneous in the weight.
To go further, we need the theory of inhomogeneous Koszul duality for operads \cite{LV}.
The Koszul dual cooperad will be a differential graded cooperad with a nonzero differential.


\subsection{A commutative diagram: isomorphism of reconfigured operads}

A conjecture relating the polynomial identities produced by the KP algorithm (duplicators) to those satisfied
by the operations produced by the BSO algorithm (diproducts) was stated in \cite{BFSO}; a year later the
conjecture was reformulated and proved by \cite{KV}.
Those papers deal exclusively with dialgebras, duplicators, and diproducts; there should be a generalization
to trialgebras, and then to pre-algebras and post-algebras.
The generalized conjecture would state, roughly speaking, that Table \ref{generalizingoperads} is a commutative
diagram of morphisms between operads.
We attempt to state the generalization as precisely as possible.

\begin{conjecture}
The following diagram of morphisms between operads commutes:
\begin{equation}
\begin{array}{ccccc@{\,}c@{\;}c}
\mathcal{P} &
\xrightarrow{\qquad \Omega \qquad} &
\mathcal{P}_\Omega &
\multicolumn{3}{c}{\xrightarrow{\qquad\qquad \cong \qquad\qquad}}
&
\mathcal{P}'_{\Omega'}
\\
\;\Bigg\downarrow\;\widehat{\,} &
\;\Bigg\Downarrow\;\widehat{\,} &
\Bigg\downarrow\;\widehat{\,}
&&&&
\Bigg\downarrow\;\widehat{\,}
\\
\widehat{\mathcal{P}} &
\xrightarrow{\qquad \widehat\Omega \qquad} &
\widehat{\mathcal{P}}_{\widehat\Omega} &
\xrightarrow{\qquad \cong \qquad} &
{\widehat{\mathcal{P}}}'_{{\widehat\Omega}'}
&
\stackrel{?}{=}
&
\widehat{\mathcal{P}'_{\Omega'}}
\end{array}
\end{equation}
The symbols in this diagram are defined as follows:
\begin{enumerate}[(1)]
\item
We start by considering operads defined in terms of operations in an operad $\mathcal{P}$.
\begin{enumerate}[(a)]
\item
Let $\mathcal{P}$ be symmetric, not necessarily binary or quadratic, and in the category $\VectF$.
\item
Let $\textbf{A}$ be the category of $\mathcal{P}$-algebras.
\item
Let $\omega_i \in \mathcal{P}(n_i)$ for $1 \le i \le m$ be (multilinear) operations in $\mathcal{P}$, which
we regard as new operations defined on the underlying vector spaces of the $\mathcal{P}$-algebras in $\textbf{A}$.
\item
For example, if $m = 1$, $n_1 = 2$ then $\omega_1$ could be the Lie bracket $ab-ba$ or the Jordan
product $ab + ba$, both of which have (skew-)symmetry.
\item
Let $\Omega = \{ \omega_1, \dots, \omega_m \}$ and let $\mathcal{P}_\Omega$ be the suboperad of $\mathcal{P}$
generated by $\Omega$.
\item
Let $R$ be the set of all relations satisfied by the operations $\Omega$ in $\mathcal{P}$.
That is, $R$ consists of the polynomial identities satisfied by the operations $\Omega$ in every $\mathcal{P}$-algebra.
\item
To separate $\mathcal{P}_\Omega$ from its embedding into $\mathcal{P}$, we make a copy,
$\Omega' = \{ \omega'_1, \dots, \omega'_m \}$: symbols representing abstract multilinear operations of
arities $n_1, \dots, n_m$.
\item
By definition, the operad $\mathcal{Q} = \mathcal{P}'_{\Omega'}$ has generators $\Omega'$ with the same symmetries
as those of $\Omega$, satisfying relations $R'$ which are copies of the relations $R$.
\item
For example, if $m = 1$, $n_1 = 2$, $\omega_1 = ab-ba$ then $\mathcal{Q} = \mathcal{P}'_{\Omega'}$ is
the Lie operad, since every identity satisfied by the Lie bracket in every associative algebra is a consequence
of anticommutativity and the Jacobi identity.
\item
On the other hand, if $m = 1$, $n_1 = 2$, $\omega_1 = ab+ba$ then $\mathcal{Q} = \mathcal{P}'_{\Omega'}$
is not the Jordan operad: the Glennie identities are satisfied by the Jordan product in every associative
algebra but are not consequences of commutativity and the Jordan identity.
\item
For an integer $d \ge 1$, the operad $\mathcal{Q}_d = \mathcal{P}'_{\Omega',d}$ is defined as
$\mathcal{Q} = \mathcal{P}'_{\Omega'}$ except that $R'$ includes only (copies of) the relations in $R$
of arity $\le d$.
\item
If $m = 1$, $n_1 = 2$, $\omega_1 = ab+ba$, $4 \le d \le 7$ then
$\mathcal{Q}_d = \mathcal{P}'_{\Omega',d}$ is the Jordan operad.
\end{enumerate}
\item
Next, we assume that we have an algorithm which takes as input an operad $\mathcal{P}$, and produces as output a
``reconfigured'' operad $\widehat{\mathcal{P}}$.
For example, the KP algorithm produces, from a given category of algebras, the corresponding category
of dialgebras.
In terms of operads, this is the Manin white product with $\Perm$.
\item
Finally, we assume that we have a corresponding algorithm for operations, also denoted by hat, which
constructs $k(i)$ ``reconfigured'' operations $\widehat{\omega_{i,j_1}}, \dots, \widehat{\omega_{i,j_k(i)}}$
from each $\omega_i$.
For simplicity, we assume that $k(i)$ equals the arity of $\omega_i$.
For example, if $\omega_i$ is the Lie bracket then the BSO algorithm produces two ``reconfigured'' operations,
the left and right Leibniz products, $\omega_{i,1} = a \dashv b - b \vdash a$ and
$\omega_{i,2} = a \vdash b - b \dashv a$.
\end{enumerate}
The conjecture amounts briefly to the statement that the ``hat'' and the ``prime'' commute.
\end{conjecture}



\end{document}